\documentclass[12pt]{amsart}
\usepackage{amssymb,amsmath,color,bm}
\usepackage{dsfont}

\usepackage{enumitem,amsrefs,hyperref}
\usepackage{color}
\usepackage{tikz}
\usepackage{pgf}
\usepackage{pgfplots}
\pgfplotsset{compat=1.10}
\usepgfplotslibrary{fillbetween}
\usepackage{bm}
\usepackage{tkz-graph}
\usetikzlibrary{shapes.geometric}
\usetikzlibrary{graphs,graphs.standard,quotes}
\usetikzlibrary{calc, positioning}
\usepackage{fullpage}
\usepackage{accents}
\usepackage[normalem]{ulem}

\newdimen\slantmathcorr
\def\oversl#1{
\setbox0=\hbox{$#1$}
\slantmathcorr=\wd0
\hskip 0.2\slantmathcorr \overline{\hbox to 0.8\wd0{%
\vphantom{\hbox{$#1$}}}}
\hskip-\wd0\hbox{$#1$}
}

\def\undersl#1{
\setbox0=\hbox{$#1$}
\slantmathcorr=\wd0
\underline{\hbox to 0.8\wd0{%
\vphantom{\hbox{$#1$}}}}
\hskip-0.8\wd0\hbox{$#1$}
}

\newtheorem{theorem}{Theorem}[section]

\newtheorem{lemma}[theorem]{Lemma}
\newtheorem{corollary}[theorem]{Corollary}

\theoremstyle{definition}
\newtheorem{definition}[theorem]{Definition}
\newtheorem{example}[theorem]{Example}

\newtheorem{question}[theorem]{Question}
\theoremstyle{remark}
\newtheorem{remark}[theorem]{Remark}

\newcommand{\RR}{\mathbb{R}}

\newcommand \U {\mathcal{U}}
\newcommand \mS {\mathcal{S}}
\newcommand \ubf {\mathbf{u}}

\newcommand{\NN}{\mathbb{N}}

\newcommand{\A}{{\mathcal{A}}}

\def \trop {\operatorname{trop}}

\newcommand{\G}{{\mathcal{G}}}

\newcommand{\Htwo}[3]{
\begin{tikzpicture}[baseline={([yshift=-.5ex]current bounding box.center)}]
\node[fill=black, circle, inner sep=1pt, minimum size=0.1cm, label=right:\footnotesize{#1}] (1) {};
\node[fill=black, circle, inner sep=1pt, minimum size=0.1cm, label=left:\footnotesize{#2}] (2) [below left = 0.3cm of 1] {};
\node[fill=black, circle, inner sep=1pt, minimum size=0.1cm, label=right:\footnotesize{#3}] (3) [below right =0.3cm of 1] {};
\draw (3)--(1)--(2);
\end{tikzpicture}}

\newcommand{\uHtwo}[0]{
\begin{tikzpicture}[baseline={([yshift=-.5ex]current bounding box.center)}]
\node[fill=black, circle, inner sep=1pt, minimum size=0.1cm] (1) {};
\node[fill=black, circle, inner sep=1pt, minimum size=0.1cm] (2) [below left = 0.3cm of 1] {};
\node[fill=black, circle, inner sep=1pt, minimum size=0.1cm] (3) [below right =0.3cm of 1] {};
\draw (3)--(1)--(2);
\end{tikzpicture}\hspace{0.1cm} }

\newcommand{\Hthree}[3]{
\begin{tikzpicture}[baseline={([yshift=-.5ex]current bounding box.center)}]
\node[fill=black, circle, inner sep=1pt, minimum size=0.1cm, label=right:\footnotesize{#1}] (1) {};
\node[fill=black, circle, inner sep=1pt, minimum size=0.1cm, label=left:\footnotesize{#2}] (2) [below left = 0.3cm of 1] {};
\node[fill=black, circle, inner sep=1pt, minimum size=0.1cm, label=right:\footnotesize{#3}] (3) [below right = 0.3cm of 1] {};
\draw (1)--(3)--(2)--(1);
\end{tikzpicture}}

\newcommand{\uHthree}[0]{
\begin{tikzpicture}[baseline={([yshift=-.5ex]current bounding box.center)}]
\node[fill=black, circle, inner sep=1pt, minimum size=0.1cm] (1) {};
\node[fill=black, circle, inner sep=1pt, minimum size=0.1cm] (2) [below left = 0.3cm of 1] {};
\node[fill=black, circle, inner sep=1pt, minimum size=0.1cm] (3) [below right = 0.3cm of 1] {};
\draw (1)--(3)--(2)--(1);
\end{tikzpicture}\hspace{0.1cm} }

\newcommand{\uthreestar}[0]{
\begin{tikzpicture}[baseline={([yshift=-.5ex]current bounding box.center)}]
\node[fill=black, circle, inner sep=1pt, minimum size=0.1cm] (1) {};
\node[fill=black, circle, inner sep=1pt, minimum size=0.1cm] (2) [above = 0.3cm of 1] {};
\node[fill=black, circle, inner sep=1pt, minimum size=0.1cm] (3) [below left = 0.3cm of 1] {};
\node[fill=black, circle, inner sep=1pt, minimum size=0.1cm] (4) [below right = 0.3cm of 1] {};
\draw (2)--(1)--(3)--(1)--(4);
\end{tikzpicture}}

\newcommand{\vedge}[2]{
\begin{tikzpicture}[baseline={([yshift=-.5ex]current bounding box.center)}]
\node[fill=black, circle, inner sep=1pt, minimum size=0.1cm,label=right:\footnotesize{#1}] (1) {};
\node[fill=black, circle, inner sep=1pt, minimum size=0.1cm,label=right:\footnotesize{#2}] (2) [below = 0.3cm of 1] {};
\draw (2)--(1);
\end{tikzpicture}}

\newcommand{\uvedge}[0]{
\begin{tikzpicture}[baseline={([yshift=-.5ex]current bounding box.center)}]
\node[fill=black, circle, inner sep=1pt, minimum size=0.1cm] (1) {};
\node[fill=black, circle, inner sep=1pt, minimum size=0.1cm] (2) [below = 0.3cm of 1] {};
\node[draw=none, fill=none, circle, inner sep=1pt, minimum size=0.1cm] (4) [left = 0.05cm of 1] {};
\draw (2)--(1);
\end{tikzpicture}\hspace{0.1cm} }

\newcommand{\pthree}[4]{
\begin{tikzpicture}[baseline={([yshift=-.5ex]current bounding box.center)}]
\node[fill=black, circle, inner sep=1pt, minimum size=0.1cm, label=left:\footnotesize{#1}] (1) {};
\node[fill=black, circle, inner sep=1pt, minimum size=0.1cm, label=left:\footnotesize{#2}] (2) [above = 0.3cm of 1] {};
\node[fill=black, circle, inner sep=1pt, minimum size=0.1cm, label=right:\footnotesize{#3}] (3) [right = 0.3cm of 2] {};
\node[fill=black, circle, inner sep=1pt, minimum size=0.1cm, label=right:\footnotesize{#4}] (4) [below = 0.3cm of 3] {};
\draw (1)--(2)--(3)--(4);
\end{tikzpicture}}

\newcommand{\upthree}[0]{
\begin{tikzpicture}[baseline={([yshift=-.5ex]current bounding box.center)}]
\node[fill=black, circle, inner sep=1pt, minimum size=0.1cm] (1) {};
\node[fill=black, circle, inner sep=1pt, minimum size=0.1cm] (2) [above = 0.3cm of 1] {};
\node[fill=black, circle, inner sep=1pt, minimum size=0.1cm] (3) [right = 0.3cm of 2] {};
\node[fill=black, circle, inner sep=1pt, minimum size=0.1cm] (4) [below = 0.3cm of 3] {};
\draw (1)--(2)--(3)--(4);
\end{tikzpicture}\hspace{0.1cm} }

\newcommand{\smallupthree}[0]{
\begin{tikzpicture}[baseline={([yshift=-.5ex]current bounding box.north)}]
\node[fill=black, circle, inner sep=1pt, minimum size=0.1cm] (1) {};
\node[fill=black, circle, inner sep=1pt, minimum size=0.1cm] (2) [above = 0.13cm of 1] {};
\node[fill=black, circle, inner sep=1pt, minimum size=0.1cm] (3) [right = 0.13cm of 2] {};
\node[fill=black, circle, inner sep=1pt, minimum size=0.1cm] (4) [below = 0.13cm of 3] {};
\draw (1)--(2)--(3)--(4);
\end{tikzpicture}}

\newcommand{\pfour}[5]{
\begin{tikzpicture}[baseline={([yshift=-.5ex]current bounding box.center)}]
\node[fill=black, circle, inner sep=1pt, minimum size=0.1cm, label=below:\footnotesize{#1}] (1) {};
\node[fill=black, circle, inner sep=1pt, minimum size=0.1cm, label=above:\footnotesize{#2}] (2) [above = 0.3cm of 1] {};
\node[fill=black, circle, inner sep=1pt, minimum size=0.1cm, label=above:\footnotesize{#3}] (3) [right = 0.3cm of 2] {};
\node[fill=black, circle, inner sep=1pt, minimum size=0.1cm, label=below:\footnotesize{#4}] (4) [below = 0.3cm of 3] {};
\node[fill=black, circle, inner sep=1pt, minimum size=0.1cm, label=below:\footnotesize{#5}] (5) [right = 0.3cm of 4] {};
\draw (1)--(2)--(3)--(4)--(5);
\end{tikzpicture}}

\newcommand{\upfour}[0]{
\begin{tikzpicture}[baseline={([yshift=-.5ex]current bounding box.center)}]
\node[fill=black, circle, inner sep=1pt, minimum size=0.1cm] (1) {};
\node[fill=black, circle, inner sep=1pt, minimum size=0.1cm] (2) [above = 0.3cm of 1] {};
\node[fill=black, circle, inner sep=1pt, minimum size=0.1cm] (3) [right = 0.3cm of 2] {};
\node[fill=black, circle, inner sep=1pt, minimum size=0.1cm] (4) [below = 0.3cm of 3] {};
\node[fill=black, circle, inner sep=1pt, minimum size=0.1cm] (5) [right = 0.3cm of 4] {};
\draw (1)--(2)--(3)--(4)--(5);
\end{tikzpicture}}

\newcommand{\smallupfour}[0]{
\begin{tikzpicture}[baseline={([yshift=-.5ex]current bounding box.north)}]
\node[fill=black, circle, inner sep=1pt, minimum size=0.1cm] (1) {};
\node[fill=black, circle, inner sep=1pt, minimum size=0.1cm] (2) [above = 0.13cm of 1] {};
\node[fill=black, circle, inner sep=1pt, minimum size=0.1cm] (3) [right = 0.13cm of 2] {};
\node[fill=black, circle, inner sep=1pt, minimum size=0.1cm] (4) [below = 0.13cm of 3] {};
\node[fill=black, circle, inner sep=1pt, minimum size=0.1cm] (5) [right = 0.13cm of 4] {};
\draw (1)--(2)--(3)--(4)--(5);
\end{tikzpicture}}

\newcommand{\ulongbroom}[0]{
\begin{tikzpicture}[baseline={([yshift=-.5ex]current bounding box.center)}]
\node[fill=black, circle, inner sep=1pt, minimum size=0.1cm, label] (1) {};
\node[fill=black, circle, inner sep=1pt, minimum size=0.1cm, label] (2) [above left= 0.3cm of 1] {};
\node[fill=black, circle, inner sep=1pt, minimum size=0.1cm, label] (3) [below left = 0.3cm of 1] {};
\node[fill=black, circle, inner sep=1pt, minimum size=0.1cm,label] (4) [right = 0.3cm of 1] {};
\node[fill=black, circle, inner sep=1pt, minimum size=0.1cm,label] (5) [right = 0.3cm of 4] {};
\node[fill=black, circle, inner sep=1pt, minimum size=0.1cm,label] (6) [right = 0.3cm of 5] {};
\draw (2)--(1)--(3)--(1)--(4)--(5)--(6);
\end{tikzpicture}}

\newcommand{\smallulongbroom}[0]{
\begin{tikzpicture}[baseline={([yshift=-.5ex]current bounding box.center)}]
\node[fill=black, circle, inner sep=1pt, minimum size=0.1cm, label] (1) {};
\node[fill=black, circle, inner sep=1pt, minimum size=0.1cm, label] (2) [above right= 0.13cm of 1] {};
\node[ fill=black, circle, inner sep=1pt, minimum size=0.1cm, label] (3) [above left = 0.13cm of 2] {};
\node[fill=black, circle, inner sep=1pt, minimum size=0.1cm,label] (4) [right = 0.13cm of 2] {};
\node[fill=black, circle, inner sep=1pt, minimum size=0.1cm,label] (5) [right = 0.13cm of 4] {};
\node[fill=black, circle, inner sep=1pt, minimum size=0.1cm,label] (6) [right = 0.13cm of 5] {};
\draw (1)--(2)--(3)--(2)--(4)--(5)--(6);
\end{tikzpicture}}

\date{\today}

\title{Tropicalization of Graph Profiles}
\thanks{Grigoriy Blekherman was partially supported by NSF grant DMS-1352073.  Annie Raymond was partially supported by NSF grant DMS-2054404. Mohit Singh was partially supported by NSF grant CCF-1717947. Rekha Thomas was partially supported by NSF grant DMS-1719538.}

\author{Grigoriy Blekherman}
\address{School of Mathematics, Georgia Institute of Technology,
686 Cherry Street
Atlanta, GA 30332}\email{greg@math.gatech.edu}

\author{Annie Raymond}
\address{Department of Mathematics and Statistics,
Lederle Graduate Research Tower, 1623D,
University of Massachusetts Amherst
710 N. Pleasant Street
Amherst, MA 01003} \email{raymond@math.umass.edu}

\author{Mohit Singh}
\address{H. Milton Stewart School of
Industrial and Systems Engineering, Georgia Institute of Technology,
755 Ferst Drive, NW, Atlanta, GA 30332}
\email{mohitsinghr@gmail.com}

\author{Rekha R. Thomas}
\address{Department of Mathematics, University of Washington, Box
  354350, Seattle, WA 98195, USA} \email{rrthomas@uw.edu}

\begin{document}

\begin{abstract}
A graph profile records all possible densities of a fixed finite set of graphs. Profiles can be extremely complicated; for instance the full profile of any triple of connected graphs is not known, and little is known about hypergraph profiles. We introduce the \textit{tropicalization} of graph and hypergraph profiles. Tropicalization is a well-studied operation in algebraic geometry, which replaces a variety (the set of real or complex solutions to a finite set of algebraic equations) with its ``combinatorial shadow". We prove that the tropicalization of a graph profile is a closed convex cone, which still captures interesting combinatorial information. We explicitly compute these tropicalizations for arbitrary sets of complete and star hypergraphs. We show they are rational polyhedral cones even though the corresponding profiles are not even known to be semialgebraic in some of these cases. We then use tropicalization to prove strong restrictions on the power of the sums of squares method, equivalently Cauchy-Schwarz calculus, to test (which is weaker than certification) the validity of graph density inequalities. In particular, we show that sums of squares cannot test simple binomial graph density inequalities, or even their approximations. Small concrete examples of such inequalities are presented, and include the famous Blakley-Roy inequalities for paths of odd length. As a consequence, these simple inequalities cannot be written as a rational sum of squares of graph densities.
\end{abstract}

\maketitle


\section{Introduction}
An important tool in the study of very large graphs is to randomly sample a fixed number of small subgraphs. This methodology goes under various
names such as {\em property testing} \cite{GGRpropertytesting} and {\em subgraph sampling} \cite[Section 1.3.1]{LovaszBook}, and the sampling statistics are in terms of densities of the subgraphs in the given graph.
 There are various notions of densities. In this paper we focus on homomorphism densities, but we say a few words at the very end on consequences for other densities.
 A graph $G$ has vertex set $V(G)$ and edge set $E(G)$, and is assumed to be simple, without loops or multiple edges. The \emph{homomorphism density} of a graph $H$ in a graph $G$, denoted by $t(H;G)$, is the probability that a random map from $V(H)$ to $V(G)$ is a graph homomorphism, i.e., it maps every edge of $H$ to an edge of $G$.

The {\em graph profile} of a collection of connected graphs $\mathcal{U} = \{C_1, \ldots, C_s \}$, denoted as $\mathcal{G}_\mathcal{U}$,
 is the closure of the set of all vectors $(t(C_1;G), t(C_2;G), \ldots, t(C_s;G))$ as $G$ varies over all graphs.
 For example, the graph profile of $\mathcal{U} = \left\{ \uvedge, \,\, \uHthree \right\}$ is
the well-known set in $[0,1]^2$ shown in Figure~\ref{fig:edge-triangle-profile1} (slightly distorted to better show its features) \cite{RazTriangle}.

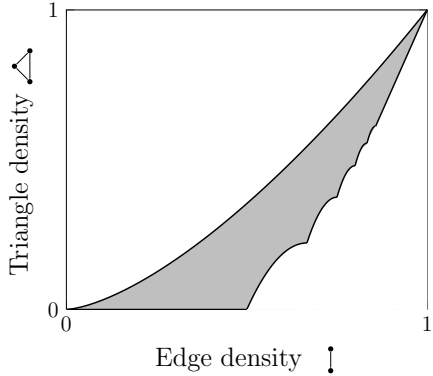
\begin{figure}[ht]
  \centering
\begin{tikzpicture}[scale=0.7]
\begin{axis}[xlabel=\large{Edge density \uvedge}, ylabel=\large{Triangle density \uHthree}, xmin=0, ymin=0, xmax=1, ymax=1, xtick={0,1}, ytick={0,1}, samples=50]
\addplot[fill=gray!50, draw=none, domain=0:1] {x^(3/2)} \closedcycle;
\addplot[thick, samples=50,domain=0:1] {x^(3/2)};
\addplot[fill=white, draw=none, domain=0:1] {-8*(x-2/3)^2+2/9} \closedcycle;
\addplot[thick, samples=30,domain=1/2:2/3] {-8*(x-2/3)^2+2/9};
\addplot[fill=white, draw=none, domain=0:1]    {-22*(x-3/4)^2 +0.375} \closedcycle;
\addplot[thick, samples=30,domain=2/3:3/4] {-22*(x-3/4)^2 +0.375};
\addplot[fill=white, draw=none, domain=0:1]   {-42*(x-4/5)^2+0.48} \closedcycle;
\addplot[thick, samples=30,domain=3/4:4/5] {-42*(x-4/5)^2+0.48};
\addplot[fill=white, draw=none, domain=0:1]  {-68*(x-5/6)^2+5/9} \closedcycle;
\addplot[thick, samples=30,domain=4/5:5/6] {-68*(x-5/6)^2+5/9};
\addplot[fill=white, draw=none, domain=0:1] {-100*(x-6/7)^2+0.6122}  \closedcycle;
\addplot[thick, samples=30,domain=5/6:6/7] {-100*(x-6/7)^2+0.6122};
\addplot[fill=white, draw=none, domain=0:1] {2.7146*x-1.7146} \closedcycle;
\addplot[thick, samples=30, domain=6/7:1] {2.7146*x-1.7146};
\addplot[thick, samples=30, domain=0:1/2] {0};
\end{axis}
\end{tikzpicture}
\caption{\label{fig:edge-triangle-profile1} The graph profile of edge and triangle.}
\end{figure}

Graph profiles are extremely complicated sets and they have been fully understood in very few cases. The study of graph profiles was initiated in \cite{MR538044}, where it was shown that a graph profile is a closed full-dimensional subset of $[0,1]^s$ for an arbitrary $s$-tuple of connected graphs. However to this day, there is no triple of connected graphs for which the graph profile is fully known.
For pairs of graphs, the profile $\mathcal{G}_\mathcal{U}$ for
$\mathcal{U}= \{ \uvedge, K_n\}$ where $K_n$ denotes the complete graph on $n$ vertices was determined first for $n=3$ in \cite{RazTriangle}, for $n=4$ in \cite{Nikiforov}, and for a general $n$ in \cite{MR3549620}. Determining the profile of $\{ \uvedge, H \}$ where $H$ is an arbitrary bipartite graph would involve resolving the famous {\em Sidorenko conjecture} which says $t(H;G) \geq t(\uvedge;G)^{|E(H)|}$. Despite considerable attention, this conjecture is only known for some classes of bipartite graphs~\cites{MR2738996,MR3456171,szegedy2014information,MR3893193}.
 Some two-dimensional projections of $\mathcal{G}_\mathcal{U}$ where
$\mathcal{U}=\{ \uvedge, \uHtwo , \uHthree\}, $ were described in \cites{MR3200287, glebov2016densities}.
As is evident from Figure~\ref{fig:edge-triangle-profile1}, graph profiles are not necessarily convex or semialgebraic sets.

In this paper we introduce the \emph{tropicalization} of graph profiles.  Tropicalization is a very well-studied operation in real and complex algebraic geometry, which replaces a variety (the set of real or complex solutions to a finite collection of algebraic equations) with its ``combinatorial shadow" \cite{MR2137980,MR3287221}. Tropicalization of real semialgebraic sets has not been explored in as much detail \cites{Alessandrini,allamigeon2020tropical, MR3314099}. As we describe below, while tropicalization loses a lot of information about a graph profile, it also keeps many of its interesting combinatorial properties.

\subsection{Tropicalization of graph profiles}
The first set of results in this paper show that even though $\mathcal{G}_\mathcal{U}$ can be very complicated \cite{MR3912211},
its tropicalization denoted as $\textup{trop}(\mathcal{G}_\mathcal{U})$ is, relatively speaking, rather simple.
For a set $\mathcal{S}  \subseteq \RR^s_{\geq 0}$, let $\log(\mathcal{S})$ denote the image of $\mathcal{S} \cap \RR^s_{>0}$ under the
map $\mathbf{v} \mapsto (\log_e v_1, \ldots, \log_e v_s)$. The tropicalization of $\mathcal{S}$, also known
as its {\em logarithmic limit set}, is
$$ \textup{trop}(\mathcal{S}) = \lim_{t \rightarrow 0} \log_{\frac{1}{t}} (\mathcal{S}).$$
It was shown in \cite{Alessandrini} that $\textup{trop}(\mathcal{S})$ is a closed cone, but it is not necessarily convex.
We prove in Theorem~\ref{thm:tropGU} that $\textup{trop}(\mathcal{G}_\mathcal{U})$
is a closed convex cone that coincides with the closure of the conical hull of $\log_e (\mathcal{G}_\mathcal{U})$.
This result extends beyond graph profiles to hypergraph profiles in Theorem~\ref{thm:tropHGU}, and in fact, to any set that has the
{\em Hadamard property}, namely that if $\bf{u}, \bf{v}$ are in the set then so is their coordinate-wise Hadamard product.

\medskip

{\em Hypergraph profiles.} In Section \ref{sec:examples} we compute the tropicalization of
an arbitrary $k$-tuple of complete graphs (and complete uniform hypergraphs) [Theorem \ref{lem:trop-clique}] , and an arbitrary $k$-tuple of star graphs (and uniform star hypergraphs) [Theorem \ref{lem:trop-stars}]. Both tropicalizations are rational polyhedral cones. This is in sharp contrast to the true profile of even any triple of such graphs being quite out of reach at the moment.

\medskip

{\em Binomial graph density inequalities.}
The cone $\textup{trop}(\mathcal{G}_\mathcal{U})$ provides a perfect framework in which to study {\em pure binomial
graph density inequalities} of the form $t(H_1;G) \geq t(H_2;G)$ where $H_1$ and $H_2$ are graphs whose connected components are contained in $\mathcal{U}=\{C_1,\dots,C_s\}$. 
Under the log map, the inequality $t(H_1;G) \geq t(H_2;G)$
becomes a linear inequality in the densities of its connected components.
The extreme rays of the dual cone $\operatorname{trop}(\mathcal{G}_\mathcal{U})^*$ generate all of the pure binomial inequalities valid on
$\mathcal{G}_{\mathcal{U}}$.
As mentioned already,  for a pair $\mathcal{U}=\{\uvedge, H\}$ where $H$ is bipartite,
$\textup{trop}(\mathcal{G}_\mathcal{U})$ captures the Sidorenko conjecture for $H$. In this case, $\textup{trop}(\mathcal{G}_\mathcal{U})$ is a two-dimensional cone in $\RR^2_{\leq 0}$, and one of its extreme rays for arbitrary $H$ is determined in \cite{blekherman2020threshold}.
Determining the other would resolve the Sidorenko conjecture.

\medskip

\subsection{Tropicalization and Sums of Squares}
Our next set of results uses tropicalization to show strong limitations for the {\em sums of squares (sos)} method, also known as
{\em Cauchy-Schwarz calculus} ~\cites{MR3002572,LovaszBook, MR2680226, Razborov07, MR3007147}, to prove graph density inequalities.

A finite $\RR$-linear combination of graphs $H_1, \ldots, H_s$, $a = \sum \alpha_i H_i$, is called a {\em graph combination}. The evaluation of a
graph combination $a$ on a graph $G$ is $a(G)=\sum \alpha_i t(H_i;G)$, and $a$ is said to be nonnegative, written as $a \geq 0$, if $\sum \alpha_i t(H_i;G) \geq 0$ for every $G$.
Equivalently, $a \geq 0$ on the graph profile $\mathcal{G}_\mathcal{U}$ where
$\mathcal{U} = \{C_1, \ldots, C_m\}$, where $C_1,\dots,C_m$ are the connected components of $H_1,\dots,H_s$. A graph combination $a$ is a {\em sum of squares (sos)} if $a = \sum [[a_j^2]]$ where $a_j$ is a graph combination of {\em partially labeled graphs}.

A natural certificate of nonnegativity of a graph combination is a sos expression for it, and {\em semidefinite programming}
can be used to search for a sos expression. It was shown in ~\cite{lovaszszegedy} that every true inequality between homomorphism
densities is a limit of Cauchy-Schwarz (sos) inequalities. Problem 17 in \cite{LovaszOpenProblems} asked whether every nonnegative
graph combination is a sos and in particular, whether the {\em Blakley-Roy inequalities}, $P_k \geq  \uvedge^k$,
where $P_k$ is a path of odd length $k$, can be certified by sos. Problem 21 in \cite{LovaszOpenProblems} asked whether every
nonnegative graph combination $a$ can be multiplied by a combination of the form $(1+b)$ where $b$ is sos so that the product is sos.
This would certify the nonnegativity of $a$.

It was shown in \cite{HN11} that the problem of verifying the validity of a polynomial inequality (or equivalently of a linear inequality) between
homomorphism densities is undecidable.
Moreover, they gave explicit examples of nonnegative graph combinations that are not sos answering the first part of Problem 17 in \cite{LovaszOpenProblems}.
The class of graph combinations that become sos after multiplication with an sos
were called {\em rational sums of squares} in \cite{HN11} after Hilbert's 17th problem. The undecidability result was used to show that there exist nonnegative graph combinations that are not rational sos, thus also solving Problem 21, although no explicit example of such graph combinations was presented. In \cite{BRST18}, we found small explicit graph density inequalities
that cannot be written as a sos, and also do not become sos after multiplication by expressions of the
form $1+b$ where $b$ is sos.  A concrete instance of our results is the family of Blakley-Roy
inequalities, $P_k \geq \uvedge^k$, for odd $k$, answering Problem 21 and the second part of Problem 17 in \cite{LovaszOpenProblems}.

We introduce a new notion of \textit{sos-testable} graph combinations, which are more general than sums of squares and rational sums of squares. Roughly speaking, sos-testable graph combinations correspond to graph combinations whose nonnegativity can be recognized by sums of squares, although there is no explicit certificate of nonnegativity. We find large families of pure binomial graph density inequalities that are not sos-testable, and even their pure binomial
approximations remain not sos-testable. These families include Blakley-Roy inequalities for odd paths.

\medskip

{\em Sos profiles and sos-testable functions.}
For a fixed positive integer $d$, define the {\em $d$-sos-profile},  denoted as $\mathcal{S}_d$, to be the set of all
points on which all sos graph combinations $\sum [[a_j^2]]$, with all $a_j$ having at most $d$ edges in their constituent graphs, are nonnegative.
Let the $(\mathcal{U},d)$-sos profile $\mathcal{S}_{\mathcal{U},d}$ be the projection on $\mathcal{S}_d$ onto the graphs in $\mathcal{U}$. We prove that $\mathcal{S}_d$
is a basic, closed semialgebraic set, and that its tropicalization $\textup{trop}(\mathcal{S}_d)$
is a rational polyhedral cone described explicitly in Theorem \ref{thm:sos-2by2}.

A graph combination $a$ is {\em sos-testable} if it is nonnegative on $\mathcal{S}_d$ for some $d$.
Sos-testable functions do not have to come with an explicit certificate of nonnegativity on an sos-profile. However,
in principle, since $\mathcal{S}_d$ is a semialgebraic set, nonnegativity of a graph combination on $\mathcal{S}_d$ can be verified via real quantifier elimination.  We show in Theorem \ref{thm:mult} that if a graph combination $a$ becomes sos-testable after multiplication by an sos-testable graph combination $b$,
then $a$ was already sos-testable.
The class of sos-testable functions includes sums of squares and also rational sums of squares, but is quite likely significantly larger.
It is not clear at this point whether even rational sos is a bigger class than just sos.

\medskip

{\em Limitations of sos.} In Section~\ref{sec:limitations} we exhibit concrete families of binomial graph density inequalities that are not sos-testable, even approximately
(Theorem~\ref{thm:noksos}).  Namely, if $\oversl{H}$ and $\undersl{H}$ are two graphs with the same number of edges where $\oversl{H}$ is a \textit{trivial square} (see Section ~\ref{sec:limitations} for details) in which every vertex has degree $p$ or $p+1$, and the maximum degree in $\undersl{H}$ is at most $p+1$, then $\oversl{H}-\undersl{H}$ is not sos-testable. An example of such a binomial inequality would be $\upthree-\uvedge^3 \geq 0$ (and in fact all Blakley-Roy inequalities for odd paths).
Even more, $\oversl{H}^k-\undersl{H}^{k+1}$ is not sos-testable for $k \geq 2|E(\oversl{H})|+1$. In particular, $\upthree^7-\uvedge^{24}$ is not sos-testable and thus cannot be written as a rational sos.
The existence of non-sos-testable
graph density inequalities follow from the undecidability result in \cite{HN11}. However, they do not provide explicit examples, and our non-approximation results are new.

\medskip

Nonnegative graph combinations admit a {\em Positivstellensatz}: any graph combination strictly positive on a graph
profile $\mathcal{G}_{\mathcal{U}}$ is a
sos \cites{lovaszszegedy, MR3272089}. It follows from this that the graph profile $\mathcal{G}_{\mathcal{U}}$ is the intersection of
the $(\mathcal{U},d)$-sos profiles for all $d$:
$$\mathcal{G}_\mathcal{U} = \bigcap_d \mathcal{S}_{(\mathcal{U},d)}.$$

Perhaps surprisingly, tropicalizations of graph and sos-profiles behave rather differently in that
$\operatorname{trop} (\mathcal{S}_{(\mathcal{U},d)})$  need not approach $\operatorname{trop}( \mathcal{G}_{\mathcal{U}})$
as $d \rightarrow \infty$.
This is because tropicalizations of graph and sos-profiles only depend on
an arbitrarily small neighborhood of the origin in the original set by Lemma  \ref{lem:tropS}, and sets may approach each other, while their neighborhoods of the origin do not, see for instance Example \ref{ex:tropdifferent}.
This phenomenon plays an important role in our ability to use tropicalizations to find non-sos-testable functions.
It enables us to find graph density inequalities that are not valid on any $d$-sos-profile while being valid on $\mathcal{G}_\mathcal{U}$.

\subsection{Open Questions.}  We now state some open questions raised by our results. All tropicalizations of graph profiles computed in Section ~\ref{sec:examples} are rational polyhedral cones. Therefore it is natural to ask the following:

 \begin{question}
 Is  $\textup{trop}(\mathcal{G}_\mathcal{U})$ a polyhedral cone for any collection $\mathcal{U}$ of connected graphs? If yes, then is it necessarily a rational polyhedral cone?
 \end{question}

It was shown in \cite{HN11} that the problem of deciding the validity of a polynomial inequality (or equivalently of a linear inequality) between
homomorphism densities is undecidable. This is equivalent to saying that verifying the validity of a polynomial inequality on a graph profile is undecidable.
Tropicalizations of graph profiles only carry information about pure binomial inequalities, and tropicalizations appear to be simpler than the full profile. Therefore we ask the following:

 \begin{question}
 Given two (not necessarily connected) graphs  $G_1$ and $G_2$ is the question of whether $G_1-G_2\geq 0$ is a valid homomorphism density inequality decidable?
 \end{question}

 \noindent The above question is equivalent to understanding whether a given integer (or rational) point lies in the dual cone $\operatorname{trop}(G_{\mathcal{U}})^*$, where $\mathcal{U}$ is the set of connected components of $G_1$ and $G_2$.

\subsection{Organization of this paper} In Section~\ref{sec:GraphProfiles} we study the tropicalization $\textup{trop}(\mathcal{S})$
 of a set $\mathcal{S} \subseteq \RR^s_{\geq 0}$ with the Hadamard property.
We prove in Lemma~\ref{lem:tropS} that $\textup{trop}(\mathcal{S})$ is a closed convex cone that coincides with the
closure of the conical hull of $\log_e(\mathcal{S})$ and that if the all-ones vector is present in $\mathcal{S}$, then $\textup{trop}(\mathcal{S})$ also
coincides with the closure of the convex hull of $\log_e(\mathcal{S})$.
Theorem~\ref{thm:tropGU} applies these results to graph profiles, proving that
$\textup{trop}(\mathcal{G}_\mathcal{U})$ is a closed convex cone that coincides with both the conical and convex hull of $\log
(\mathcal{G}_\mathcal{U})$.

Lemma~\ref{lem:tropS} can also be applied to hypergraph profiles (Theorem~\ref{thm:tropHGU}). In Section~\ref{sec:examples} we compute the
tropicalizations of the profiles of an arbitrary number of hypergraphs from two families -- complete hypergraphs and star hypergraphs.
Both examples yield rational polyhedral cones that can be described explicitly. Their actual profiles are currently unknown.

We introduce the $d$-sos-profile $\mathcal{S}_d$  in Section~\ref{sec:SosProfiles} and prove that it is a basic closed semialgebraic set.
Theorem~\ref{thm:sos-2by2}  proves that $\textup{trop}(\mathcal{S}_d)$ is a rational polyhedral cone
whose inequalities can be described explicitly by the $2 \times 2$ principal minors of a symbolic matrix. Graph profiles are contained
in sos-profiles. This section also introduces the notion of
sos-testable graph density inequalities. We prove in Theorem~\ref{thm:mult}
 that if $b$ and $ab$ are sos-testable then so is $a$. In particular, if $a$ is not
sos-testable it is neither a sos nor a rational sos (Corollary~\ref{cor:not sostestable implies not rational sos}).

In Section~\ref{sec:limitations} we find explicit families of pure binomial inequalities that are not sos-testable, even approximately.


\section{Tropicalization of Graph Profiles}
\label{sec:GraphProfiles}

\begin{definition} Let $\mathcal{U} = \{C_1, \ldots, C_s\}$ be a collection of connected graphs.
The {\em graph profile} of $\mathcal{U}$, denoted as $\mathcal{G}_{\mathcal{U}}$, is the closure of the set of vectors
$(t(C_1;G), t(C_2;G), \cdots, t(C_s;G))$  as $G$ varies over all unlabeled graphs.
\end{definition}

For any $\mathcal{U}$, the graph profile $\mathcal{G}_\mathcal{U}$ is contained in $[0,1]^s$.
They are highly complicated objects and very few of them are known explicitly. Note from Figure~\ref{fig:edge-triangle-profile1}
that neither the graph profile, nor its convex hull, may be semialgebraic.

In this section we use tropical geometry
to pass from the complicated graph profile  $\mathcal{G}_\mathcal{U}$ to its {\em tropicalization},
which as we will see is a cone, and hence much easier to understand.

Let $\log \,:\, \RR^s_{>0} \rightarrow \RR^s$ be defined as $\log({\bf v}) := (\log_e(v_1), \ldots, \log_e(v_s))$. If we need to change the
base of the log from $e$ to $\alpha$ then we will explicitly write $\log_\alpha$.
For a set $\mathcal{S} \subseteq \RR^s_{\geq 0}$ we define $\log(\mathcal{S}) := \log (\mathcal{S} \cap \RR^s_{>0})$.
For any set $\mathcal{S} \subseteq \RR_{\geq 0}$ and $\tau \in (0,1)$ consider
$$ \log_{\frac{1}{\tau}}(\mathcal{S}) =
\frac{-1}{\log_e \tau} \log_e(\mathcal{S}). $$
The tropicalization of $\mathcal{S}$, which is also called the {\em logarithmic limit set} of $\mathcal{S}$, is
$$\textup{trop}(\mathcal{S}) := \lim_{\tau \rightarrow 0} \log_{\frac{1}{\tau}}(\mathcal{S}).$$
By Proposition~2.2 \cite{Alessandrini}, $\textup{trop}(\mathcal{S})$ is a closed cone in $\RR^s$. A working definition of what it
means for a point ${\bf y}$ to lie in $\textup{trop}(\mathcal{S})$ is that 
for a sequence $\tau_k\in (0, \epsilon)$  indexed by $k \in \NN$  converging to $0$ 
(equivalently, any such sequence), 
there exists a sequence ${\bf y}(k) \in \mathcal{S} \cap \RR^s_{>0}$ such that 
$\log_{\frac{1}{\tau_k}} {\bf y}(k) \rightarrow {\bf y}$ as $\tau_k \rightarrow 0$ ~\cite[Proposition~2.1]{Alessandrini}. In other words, $\textup{trop}(\mathcal{S})$ consists of all accumulations points under the map $\log_{1/\tau_k}$ applied to $\mathcal{S} \cap \RR^s_{>0}$ for all possible choices of sequences of bases $\tau_k$.
Note that since the $\log$ map is only defined on positive points, $\log_{\frac{1}{\tau_k}} {\bf y}(k)$ can exist only if ${\bf y}(k) \in \RR^s_{>0}$. For sets $\mathcal{S}\subseteq  \mathbb{R}^s_{\geq 0}$ such that $\mathcal{S}=\textup{cl}(\mathcal{S}\cap \mathbb{R}^s_{>0})$, $\textup{trop}(\mathcal{S})$ does not lose information carried by the points with zero coordinates in $\mathcal{S}$. 
The {\em Hadamard product} of ${\bf v}, {\bf w} \in \RR^s$ is defined to be
${\bf v} \cdot {\bf w} = (v_1w_1, \ldots, v_sw_s)$.
We say that a set $\mathcal{S} \subseteq \RR^s_{\geq 0}$ has the {\em Hadamard property} if for any two vectors ${\bf v}, {\bf w} \in \mathcal{S}$,
${\bf v} \cdot {\bf w} \in \mathcal{S}$. For a vector ${\bf v}$ and a positive integer $k$, define the $k$th power of ${\bf v}$ to be
${\bf v}^k := (v_1^k, \ldots, v_s^k)$.

\begin{lemma} \label{lem:tropS}
Suppose $\mathcal{S} \subseteq \RR_{\geq 0}^s$ has the Hadamard property. Then
\begin{enumerate}
\item $\textup{trop}(\mathcal{S})$ is the closure of the union of all the rays
from the origin through points in $\log(\mathcal{S})$.
\item $\textup{trop}(\mathcal{S})$ is a closed convex cone and $\textup{trop}(\mathcal{S}) = \textup{cl}(\textup{cone}(\log (\mathcal{S})))$, where $\textup{cone}(\cdot)$ denotes the conical hull.
\item If ${\bf 1} \in \mathcal{S}$, $\textup{trop}(\mathcal{S}) = \textup{cl}(\textup{conv}(\log (\mathcal{S})))$.
\item If $\mathcal{S} \subseteq [0,1]^s$ and $\mathcal{S}=\textup{cl}(\mathcal{S}\cap [0,1)^s)$, then for any $\epsilon > 0$,
$\textup{trop}(\mathcal{S})$ is determined by the nonempty neighborhood
$\mathcal{S} \cap B({\bf 0},\epsilon)$ of $\mathcal{S}$, where $B({\bf 0},\epsilon)$ is the ball with center ${\bf 0}$ and radius $\epsilon$. 
\end{enumerate}
\end{lemma}

\begin{proof}
\begin{enumerate}

\item We first show that $\textup{trop}(\mathcal{S})$ is contained in the closure of the union of all the rays from the origin through points in $\log(\mathcal{S})$. Suppose ${\bf y} \in \textup{trop}(\mathcal{S})$. Then there exists sequences ${\bf y}(k) \in \mathcal{S} \cap \RR^s_{>0}$ and $\tau_k \in (0, \epsilon)$ such that as $\tau_k \rightarrow 0$,
$$\log_{\frac{1}{\tau_k}} {\bf y}(k) = \frac{-1}{\log \tau_k} \log {\bf y}(k)  \rightarrow {\bf y}.$$ Since
$\log {\bf y}(k) \in \log(\mathcal{S})$ and $ \frac{-1}{\log \tau_k} > 0$, we get that
$\log_{\frac{1}{\tau_k}} {\bf y}(k)$ is in the union of all the rays from the origin through points in $\log(\mathcal{S})$ for all $k$, and hence, ${\bf y}$ is in their closure.

For the other inclusion, we want to show that, for any ${\bf v}\in \mathcal{S}$, the ray generated by $\log({\bf v})$ is in $\trop(\mathcal{S})$. Since $\trop(\mathcal{S})$ is a cone and changing bases just rescales $\log({\bf v})$, it suffices to show that $\log_e(\mathbf{v})\in \trop(\mathcal{S})$.   Note that $\log_{e} ({\bf v}) = \log_{e^k} ({\bf v}^k)$ for any $k\in \mathbb{N}$ and  since $\mathcal{S}$ has the Hadamard property, ${\bf v}^k$ is also in $\mathcal{S}$ for all positive integers $k$.  Thus there are sequences $\tau_k:=\frac{1}{e^k}$ and ${\bf v}(k):={\bf v}^k$ where $\tau_k \rightarrow 0$ and $\log_{\frac{1}{\tau_k}} ({\bf v}(k)) = \log_{e} ({\bf v})$, and so $\log_{e} ({\bf v}) \in \textup{trop}(\mathcal{S})$.
Since $\textup{trop}(\mathcal{S})$ is closed, we can conclude that the closure of the union of all the rays
from the origin through points in $\log(\mathcal{S})$ is contained in
$\textup{trop}(\mathcal{S})$.

\item
We already know that $\textup{trop}(\mathcal{S})$ is closed. We now show that it is a convex cone. By (1), it suffices to show that for any $\log({\bf v}), \log({\bf w})\in \log(\mathcal{S})$,  $\alpha \log({\bf v}) + \beta \log({\bf w}) \in \trop(\mathcal{S})$ for any $\alpha, \beta \geq 0$. We can assume that $\alpha, \beta \in \mathbb{Q}$ since $\trop(\mathcal{S})$ is closed, and so we can further assume that $\alpha, \beta$ have the same denominator, say $\alpha=\frac{a_1}{b}$ and $\beta=\frac{a_2}{b}$ where $a_1, a_2, b\in \mathbb{N}$. From (1), we only need to show that $a_1 \log({\bf v})+ a_2 \log({\bf w})\in \trop(\mathcal{S})$. This is equal to $\log({\bf v}^{a_1}\cdot {\bf w}^{a_2})$ which, by the Hadamard property, is contained in $\log(\mathcal{S})$, and, from (1),  thus also in the tropicalization.

From (1), we know that $\textup{trop}(\mathcal{S})$ is contained in the closure of the union of all the rays
from the origin through points in $\log(\mathcal{S})$, which is contained in $\textup{cl}(\textup{cone}(\log (\mathcal{S})))$. To show the reverse inclusion, by (1), we know $\trop(\mathcal{S})\supseteq \log(\mathcal{S})$, and since $\textup{trop}(\mathcal{S})$ is a closed convex cone, $\trop(\mathcal{S}) \supseteq \textup{cl}(\textup{cone}(\log (\mathcal{S})))$, and the result holds.

\item If ${\bf 1} \in \mathcal{S}$, then ${\bf 0} = \log {\bf 1} \in \log(\mathcal{S})$. We already saw that all positive integer multiples of
a point in $\log({\mathcal S})$ is also in $\log({\mathcal S})$. Together these facts imply that
$\textup{cl}(\textup{conv}(\log (\mathcal{S}))) = \textup{cl}(\textup{cone}(\log (\mathcal{S}))) = \textup{trop}(\mathcal{S})$.

\item Since $\textup{trop}(\mathcal{S})$ is
a closed convex cone, it is fully described by the linear inequalities that are valid on it.
A linear inequality ${\bf a}^\top{\bf y}   \geq {\bf b}^\top {\bf y}$
valid on $\log(\mathcal{S})$ corresponds to the binomial inequality ${\bf x}^{\bf a} \geq {\bf x}^{\bf b}$ on $\mathcal{S}\cap \RR^s_{>0}$.
Here we are setting $y_i = \log(x_i)$. Therefore, to prove the claim
it suffices to argue that if a binomial inequality is valid on a small neighborhood of the origin in  $\mathcal{S}\cap \RR^s_{>0}$ then it is in fact valid on all of  $\mathcal{S}\cap \RR^s_{>0}$.
Suppose there is some binomial inequality ${\bf x}^{\bf a} \geq {\bf x}^{\bf b}$ that is valid on the neighborhood and
${\bf v} \in  \mathcal{S}\cap \RR^s_{>0}$ violates it. Without loss of generality, we can assume that all components of ${\bf v}$ are less than 1. Then there is some large enough positive integer
 $k$ for which ${\bf v}^k$ lies in the neighborhood we considered. If ${\bf v}^{\bf a} < {\bf v}^{\bf b}$ then we also have
 $({\bf v}^k)^{\bf a} < ({\bf v}^k)^{\bf b}$ which is a
 contradiction.  Thus $\textup{trop}(\mathcal{S})$ is determined by the behavior of $\mathcal{S}$ near the origin.
 \end{enumerate}
\end{proof}

Note that even though two sets might converge, their tropicalizations might not as seen in the following example.
This example also highlights the role of the neighborhood of the origin as in Lemma~\ref{lem:tropS} (4).

\begin{example}\label{ex:tropdifferent}
\begin{figure}[ht]
  \centering
\resizebox{!}{3cm}
{\begin{tikzpicture}
\begin{axis}[
  axis x line=center,
  axis y line=center,
  xtick={0,1},
  ytick={0,1},
  xmin=-0.1,
  xmax=1.2,
  ymin=-0.1,
  ymax=1.2]
\addplot +[mark=none, thick, black] coordinates {(1,0) (1,1)};
\addplot[name path=A, thick, black, domain=0:1] {x};
\addplot[name path=B, thick, black, domain=0:1] {0};
\addplot[gray!50] fill between[of=A and B];
\end{axis}
\end{tikzpicture} \quad \quad \begin{tikzpicture}
\begin{axis}[
  axis x line=center,
  axis y line=center,
  xtick={0,1},
  ytick={0,0.1,1},
  yticklabels = {0,$\Large{\varepsilon}$,1},
  xmin=-0.1,
  xmax=1.2,
  ymin=-0.1,
  ymax=1.2]
\addplot +[mark=none, thick, black] coordinates {(1,0) (1,1)};
\addplot +[mark=none, thick, black] coordinates {(0,0) (0,0.1)};
\addplot[name path=A, thick, black, domain=0:1] {0.9*x+0.1};
\addplot[name path=B, thick, black, domain=0:1] {0};
\addplot[gray!50] fill between[of=A and B];
\end{axis}
\end{tikzpicture}}
\caption{\label{fig:S1S2} The sets $\mathcal{S}$ (left) and $\mathcal{S}_{\varepsilon}$ (right)}
\end{figure}
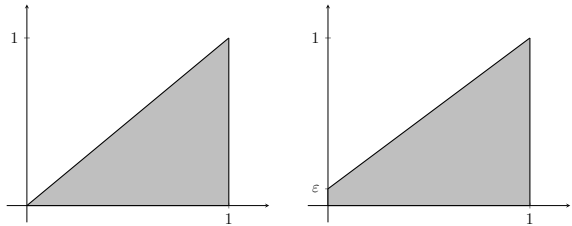
Consider the two sets in Figure \ref{fig:S1S2}. On the left is a triangle and on the right a slight
modification of the triangle into a quadrilateral with new vertex $(0,\varepsilon)$. The neighborhood of the origin
is different for the two sets.
Observe that $\mathcal{S}=\lim_{\varepsilon\rightarrow 0}\mathcal{S}_{\varepsilon}$ as is clear from Figure~\ref{fig:S1S2}. On the other hand, their tropicalizations, seen in Figure \ref{fig:tropS1S2}, do not, i.e., $\trop(\mathcal{S})\neq \lim_{\varepsilon\rightarrow 0} \trop(\mathcal{S}_\varepsilon)$ since the
the neighborhood of the origin is different in sets $\mathcal{S}$ and $\mathcal{S}_{\varepsilon}$ for any $\varepsilon>0$.

\begin{figure}[ht]
  \centering
\resizebox{!}{3cm}
{\begin{tikzpicture}
\begin{axis}[
  axis x line=center,
  axis y line=center,
  xtick={0},
  ytick={0},
  xmin=-5.5,
  xmax=1.5,
  ymin=-5.5,
  ymax=1.5]
\addplot +[mark=none, thick, black] coordinates {(0,0) (0,-5.5)};
\addplot[name path=A, thick, black, domain=-5.5:0] {x};
\addplot[name path=B, thick, black, domain=-5.5/10000:0] {10000*x};
\addplot[gray!50] fill between[of=A and B];
\end{axis}
\end{tikzpicture} \quad \quad \begin{tikzpicture}
\begin{axis}[
  axis x line=center,
  axis y line=center,
  xtick={0},
  ytick={0},
  xmin=-5.5,
  xmax=1.5,
  ymin=-5.5,
  ymax=1.5]
\addplot +[mark=none, thick, black] coordinates {(0,0) (0,-5.5)};
\addplot[name path=A, thick, black, domain=-5.5:0] {0};
\addplot[name path=B, thick, black, domain=-5.5/10000:0] {10000*x};
\addplot[name path=C, white, domain=-5.5:0] {-5.5};
\addplot[gray!50] fill between[of=A and C];
\end{axis}
\end{tikzpicture}}
\caption{\label{fig:tropS1S2} The tropicalizations of $\mathcal{S}$ (left) and $\mathcal{S}_\varepsilon$ (right)}
\end{figure}
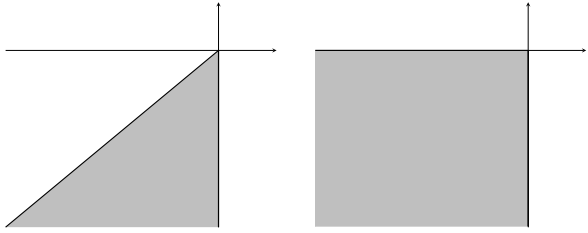

\end{example}

We now apply these results to the graph profile $\mathcal{G}_\mathcal{U}$ which is known to be
a connected and full-dimensional set \cite{MR538044}.
Moreover, it is known that every ${\bf v} \in \mathcal{G}_\mathcal{U}$ is arbitrarily close to 
$(t(C_1;G), \ldots, t(C_s;G))$ for some graph $G$.
As we will see in the proof of Theorem \ref{thm:mult}, one can argue that every neighborhood of ${\bf v}$ has a full-dimensional ball containing ${\bf v}$ that is contained in $\mathcal{G}_\mathcal{U}$. Therefore, there is a positive point in $\mathcal{G}_\mathcal{U}$
arbitrarily close to ${\bf v}$, and no information is lost by passing to $\log(\mathcal{G}_\mathcal{U})$.
Also, since $\mathcal{G}_\mathcal{U}$ is contained in $[0,1]^s$, $\log(\mathcal{G}_\mathcal{U})$ and
$\textup{trop}(\mathcal{G}_\mathcal{U})$ lie in $\RR^s_{\leq 0}$.

\begin{theorem} \label{thm:tropGU}
For the graph profile $\mathcal{G}_\mathcal{U}$,
$$\textup{trop}(\mathcal{G}_\mathcal{U}) = \textup{cl}(\textup{cone}(\log (\mathcal{G}_\mathcal{U}))) =
\textup{cl}(\textup{conv}(\log (\mathcal{G}_\mathcal{U})))
 \subseteq \RR^s_{\leq 0}.$$
Further, for any $\epsilon > 0$, $\textup{trop}(\mathcal{G}_\mathcal{U})$ is determined by the nonempty neighborhood
$\mathcal{\mathcal{G}_\mathcal{U}} \cap B({\bf 0},\epsilon)$ where $B({\bf 0},\epsilon)$ is the ball with center ${\bf 0}$ and radius $\epsilon$.
\end{theorem}

\begin{proof}
By Lemma~\ref{lem:tropS} we just need to show that $\mathcal{G}_\mathcal{U}$ has the Hadamard property and contains
${\bf 1}$.
Equation 5.30 in \cite{LovaszBook} implies that $t(C_i;G)t(C_i;G') = t(C_i; G \times G')$ where $G \times G'$ is the categorical product of
$G$ and $G'$. Therefore, if  $(t(C_1;G), \ldots, t(C_s;G)) \in \mathcal{G}_\mathcal{U}$ and $(t(C_1;G'), \ldots, t(C_s;G')) \in \mathcal{G}_\mathcal{U}$, then $(t(C_1;G)t(C_1;G'), \ldots, t(C_s;G)t(C_s;G')) \in \mathcal{G}_\mathcal{U}$.

For the sequence of complete graphs $K_n$, $(t(C;K_n) \,:\, C \in \mathcal{U}) \rightarrow {\bf 1}$ as
$n \rightarrow \infty$ which implies that ${\bf 0} = \log({\bf 1})$ lies in the closure of the convex hull of
$\log(\mathcal{G}_\mathcal{U})$.
\end{proof}

 Theorem~\ref{thm:tropGU} implies that we obtain a great simplification in structure when we pass from
 the graph profile $\mathcal{G}_\mathcal{U}$ to its tropicalization, $\textup{trop}(\mathcal{G}_\mathcal{U})$, which is a closed
 convex cone. A natural next question is the following:

 \begin{question}
 Is  $\textup{trop}(\mathcal{G}_\mathcal{U})$ a polyhedral cone for any collection $\mathcal{U}$ of connected graphs? If yes, then is it necessarily a rational polyhedral cone?
 \end{question}

 If $\mathcal{U}$ contains two graphs, then
 $\textup{trop}(\mathcal{S})$ is indeed a polyhedral cone with two extreme rays since it is a two-dimensional cone. Could the generators of the extreme rays be non-rational? No graph profile for three graphs is known. Is there a graph profile for three graphs for which $\textup{trop}(\mathcal{S})$ is not polyhedral?

\begin{example} For $\mathcal{U} = \left\{ \uvedge, \,\, \uHthree \right\}$, we saw the
graph profile $\mathcal{G}_\mathcal{U}$ in
Figure~\ref{fig:edge-triangle-profile1}.

\begin{figure}[ht]
  \centering
\resizebox{!}{4cm}
{\begin{tikzpicture}
\begin{axis}[
  axis x line=center,
  axis y line=center,
  xtick={0},
  ytick={0},
  xlabel={\uvedge},
  ylabel={\uHthree},
  xlabel style={right},
  ylabel style={above},
  xmin=-5.5,
  xmax=1.5,
  ymin=-5.5,
  ymax=1.5]
\addplot +[mark=none, thick, black] coordinates {(0,0) (0,-5.5)};
\addplot[name path=A, thick, black, domain=-11/3:0] {3*x/2};
\addplot[name path=B, thick, black, domain=-5.5/10000:0] {10000*x};
\addplot[gray!50] fill between[of=A and B];
\end{axis}
\end{tikzpicture}}
\caption{\label{fig:edge-triangle-profile} The tropicalization of the graph profile of an edge and a triangle }
\end{figure}
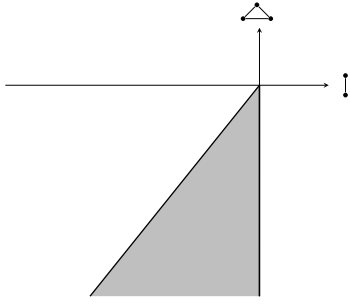

The curve bounding the upper part of the profile is $\uvedge^3 \geq \uHthree^2$ which stands for the density inequality
$t(\uvedge^3;G) \geq t(\uHthree^2;G)$, or equivalently, $t(\uvedge \uvedge \uvedge;G) \geq t(\uHthree \, \uHthree;G)$ for all unlabeled graphs $G$.
Moreover, we know that $\uvedge \leq 1$.
Let $y_1 := \log \uvedge$ and $y_2 := \log \uHthree$. Then
the above binomial inequalities correspond under the $\log$ map to the linear
inequalities
$$ 3 y_1 - 2 y_2 \geq 0 \,\,\textup{ and } \,\, y_1 \leq 0 $$
which form the cone generated by the rays $(-2,-3)$ and $(0,-1)$ in $\mathbb{R}^2_{\leq 0}$ as shown in
Figure~\ref{fig:edge-triangle-profile}. In the next section, we will see that this cone is indeed $\textup{trop}(\mathcal{G}_\mathcal{U})$.
\end{example}

\begin{remark} Note that $\textup{trop}(\mathcal{G}_\mathcal{U})$ can be understood as an object recording all possible order of growths of densities, i.e., recording whether there exists a sequence of graphs converging to that order of growth.
\end{remark}

\section{Explicit Tropicalization of Cliques and Stars}\label{sec:examples}

In this section, we explicitly compute $\trop(\G_{\U})$ for a finite collection $\U$ from two hypergraph families, cliques and stars. In both cases we will see that the tropicalizations are rational polyhedral simplicial cones.

Let $K^{(r)}_p$ be the complete $r$-uniform hypergraph on $p$ vertices, i.e., the graph on $p$ vertices where every set of $r$ vertices form a (hyper)edge. When $r=2$, $K^{(2)}_p$ is simply the complete graph on $p$ vertices. In Section~\ref{sec:example1}, we describe $\trop(\G_{\U})$ when ${\U}= \{K^{(r)}_r, K^{(r)}_{r+1}, \ldots, K^{(r)}_l\}$ is the collection of complete $r$-uniform hypergraphs for any $r\geq 2$ and $l\geq r$.

Let $S^{(r)}(b,c)$ denote the $r$-uniform hypergraph with $b$ edges all of which intersect in some set of $c$ vertices, and nowhere else. We call such graphs, \emph{stars with $b$ branches}, and we call the intersection of all the edges the \emph{center}. For example, $S^{(2)}(b,1)=K_{1,b}$, the complete bipartite graph with parts of size $1$ and $b$. Note that $S^{(r)}(b,c)$ has $b(r-c)+c$ vertices. In Section~\ref{sec:example2}, we describe $\trop(\G_{\U})$ when ${\U}= \{S^{(r)}(1,c),S^{(r)}(2,c),\ldots, S^{(r)}(l,c)\}$ for any $r\geq 2$, $c\leq r$ and $b\geq 1$.

We begin by showing that Theorem \ref{thm:tropGU} also holds for $r$-uniform hypergraphs. We first define what we mean by the product of two hypergraphs.

\begin{definition}
Let $G_1$ and $G_2$ be two $r$-uniform hypergraphs. The \emph{direct product} of $G_1$ and $G_2$ is $G_1\times G_2$ where $V(G_1\times G_2)=\{uv \,:\,  u\in V(G_1) \textrm{ and } v\in V(G_2)\}$, and $$E(G_1\times G_2)=\left\{\{u_1v_1, u_2v_2, \ldots, u_rv_r\} \,:\, \{u_1, \ldots, u_r\}\in E(G_1) \textrm{ and } \{v_1, \ldots, v_r\}\in E(G_2) \right\}.$$ Note that any pair of edges $\{u_1, \ldots, u_r\}\in E(G_1)$ and $\{v_1, \ldots, v_r\}\in E(G_2)$ gives rise to $r!$ different edges in $E(G_1\times G_2)$.
\end{definition}

\begin{theorem} \label{thm:tropHGU}
For a $r$-uniform hypergraph profile $\mathcal{G}_\mathcal{U}$, where $|\mathcal{U}| = s$, $$\textup{trop}(\mathcal{G}_\mathcal{U}) = \textup{cl}(\textup{cone}(\log (\mathcal{G}_\mathcal{U}))) = \textup{cl}(\textup{conv}(\log (\mathcal{G}_\mathcal{U})))  \subseteq \RR^s_{\leq 0}.$$ Further, for any $\epsilon > 0$, $\textup{trop}(\mathcal{G}_\mathcal{U})$ is determined by the nonempty neighborhood $\mathcal{\mathcal{G}_\mathcal{U}} \cap B({\bf 0},\epsilon)$ where $B({\bf 0},\epsilon)$ is the ball with center ${\bf 0}$ and radius $\epsilon$.
\end{theorem}

\begin{proof}
We first show that $\mathcal{G}_\mathcal{U}$ has the Hadamard property, i.e., that $$t(H;G)t(H;G')=t(H;G\times G')$$ for all $r$-uniform hypergraphs $H, G, G'$. Since the total number of maps from
$V(H)$ to $V(G \times G')$ is $(|V(G)||V(G')|)^{|V(H)|}$, the denominators on both sides of the equation are the same.
For homomorphisms $\varphi: V(H) \rightarrow V(G)$ and $\varphi': V(H) \rightarrow V(G')$, the map
 $\psi: V(H) \rightarrow V(G\times G')$ such that $\psi(v)\mapsto \varphi(v)\varphi'(v)$ is also a homomorphism.
 Conversely, for a homomorphism $\psi: V(H) \rightarrow V(G)\times V(G')$, the projections
 $\varphi = \psi_G \,:\, V(H) \rightarrow V(G)$ and $\varphi' = \psi_{G'} \,:\, V(H) \rightarrow V(G')$ onto the two components
 $V(G)$ and $V(G')$ are homomorphisms. Thus, the numerators on both sides of the equation are also the same.
Therefore, if  $(t(C_1;G), \ldots, t(C_s;G)) \in \mathcal{G}_\mathcal{U}$ and $(t(C_1;G'), \ldots, t(C_s;G')) \in \mathcal{G}_\mathcal{U}$, then $(t(C_1;G)t(C_1;G'), \ldots, t(C_s;G)t(C_s;G')) \in \mathcal{G}_\mathcal{U}$.

We now show that $\mathcal{G}_\mathcal{U}$ contains
${\bf 1}$. For the sequence of complete $r$-uniform hypergraphs $K^{(r)}_n$, $(t(C;K^{(r)}_n) \,:\, C \in \mathcal{U}) \rightarrow {\bf 1}$ as
$n \rightarrow \infty$ which implies that ${\bf 0} = \log({\bf 1})$ lies in the closure of the convex hull of
$\log(\mathcal{G}_\mathcal{U})$.

The statement of the theorem now follows from Lemma~\ref{lem:tropS}.
\end{proof}

\subsection{Tropicalization of Hypergraph Clique Profiles.}
\label{sec:example1}

Consider the hypergraph family ${\U}= \{K^{(r)}_r, K^{(r)}_{r+1}, \ldots, K^{(r)}_l\}$ for some $r\geq 2$ and $l\geq r$. Observe that both $\G_{\U}$ and $\trop(\G_{\U})$ lie in $\RR^{l-r+1}$ where the $i^{th}$ coordinate corresponds to the graph $K^{(r)}_{r+i-1}$ for
 any $1\leq i\leq l-r+1$. In Theorem~\ref{lem:trop-clique}, we describe the facets and extreme rays of $\trop(\G_\U)$. For any $\mathbf{y} \in \mathbb{R}^{l-r+1}$, denote by $y_{K^{(r)}_i}$ the coordinate of $\mathbf{y}$ indexed by $K^{(r)}_i$ for $r \leq i \leq l$. Also, define $\mathbf{1}_{K^{(r)}_i}$ for some $r \leq i \leq l$ to be the point in $\mathbb{R}^{l-r+1}$ with $1$ in the coordinate labeled by $K^{(r)}_i$ and $0$ otherwise.

\begin{theorem}\label{lem:trop-clique}
Let ${\U}= \{K^{(r)}_r, K^{(r)}_{r+1}, \ldots, K^{(r)}_l\}$, then 

$$\trop(\G_\U)=\left\{{\bf y} \in \mathbb{R}^{l-r+1} \, : \,
\begin{array}{l}
y_{K_r^{(r)}} \leq 0,\\
 (r+i)y_{K^{(r)}_{r+i-1}} -(r+i-1)y_{K^{(r)}_{r+i}} \geq 0   \,\,\,\,\forall \;1\leq i\leq l-r
 \end{array}
 \right\}.$$
Moreover, the extreme rays of $\trop(\G_\U)$ are $$\mathbf{u}_i=-\sum_{j=i}^{l-r+1} (r+j-1) \mathbf{1}_{K_{r+j-1}^{(r)}}$$ for $1\leq i\leq l-r+1$.
\end{theorem}

\begin{proof}
Let $C$ be the cone on the right hand side of the equation in the theorem.
First observe that $t(K^{(r)}_r,G) \leq 1$ is a valid inequality for all graphs $G$ and thus the
inequality $y_{K^{(r)}_r} \leq 0$ is valid for $\trop(\G_\U)$.
The Kruskal-Katona theorem~\cites{kruskal1963number,katona2009theorem} (see also \cite{MR3200287}) in the context of graph homomorphism densities of complete graphs implies that for any integers $r\leq p < q$,
$\left(t(K^{(r)}_p,G)\right)^q - \left(t(K^{(r)}_q,G)\right)^p \geq 0$ is valid for each graph $G$. This binomial inequality for
$\G_\U$ implies that the inequality $q y_{K^{(r)}_p} -p y_{K^{(r)}_q }\geq 0$ is valid for $\trop(\G_\U)$ for each $r\leq p<q$.
Using the inequalities for each $r\leq p\leq l-1$ and $q=p+1$, we obtain that all the inequalities describing $C$ are valid for $\trop(\G_\U)$.
Thus $\trop(\G_\U)\subseteq C$.

Before we prove the other containment, we show that extreme rays of $C$ are as claimed in the theorem.

\begin{lemma}
The extreme rays of $C$ are $\mathbf{u}_i=-\sum_{j=i}^{l-r+1} (r+j-1) \mathbf{1}_{K_{r+j-1}^{(r)}}$ for $i\in \{1, \ldots, l-r+1\}$.
\end{lemma}

\begin{proof}
We have $C=\{{\bf y} \in \RR^{l-r+1}: M {\bf y} \geq 0\}$ where
$$ M := \begin{bmatrix}
-1 & 0 & 0 & 0 & 0 & \cdots &0 & 0& 0  \\
(r+1) & -r & 0 & 0 & 0 & \cdots & 0 & 0& 0  \\
0 & (r+2) & -(r+1) & 0 & 0 & \cdots & 0 & 0 & 0 \\
\vdots & \vdots & \vdots & \vdots & \vdots & \vdots & \vdots&\\
0 & 0 & 0 & 0 & 0 & \cdots  &(l-1) &-(l-2) & 0 \\
0 & 0 & 0 & 0 & 0 & \cdots & 0& l &  -(l-1)
\end{bmatrix}.$$
Since $M$ is lower triangular square matrix of size $l-r+1$, with non-zero diagonal, it is invertible. The candidate extreme rays can be obtained by setting a subset of $l-r$ constraints at equality or equivalently, all but one of the constraints at equality.  Let $M_{-i}$ denote the matrix obtained after removing the $i^{\text{th}}$ row of $M$. Then a simple check shows that solutions $M_{-i} \mathbf{y} =0$ are exactly $\{\lambda \cdot \mathbf{u}_i:\lambda \in \RR\}$. Since $\mathbf{u}_i\in C$, we obtain that it is an extreme ray. Since these are all the candidate extreme rays, we have the lemma.
\end{proof}

To prove $C\subseteq \trop(\G_\U)$, it is enough to show that the extreme rays of $C$ are contained in $\trop(\G_\U)$ as shown in the next lemma.
\begin{lemma}
The vectors $\mathbf{u}_i$ are in $\textup{trop}(\mathcal{G}_{\mathcal{U}})$ for every $i\in \{1, \ldots, l-r+1\}$.
\end{lemma}

\begin{proof}
Let $T_{m,k}^{(r)}$ be a $r$-uniform hypergraph on $m$ vertices partitioned into $k$ parts as equal in size as possible where every $r$ vertices coming from $r$ different parts form an edge (when $r=2$, this graph is the Tur\'an graph on $m$ vertices with $k$ parts). If $k<r$, then $T_{m,k}^{(r)}$ is the empty graph on $m$ vertices. Note that $T_{m,k}^{(r)}$ contains cliques of size $k$ (and less), but no clique of size $k+1$ (or more).

Fix some $i \in \{1, \ldots, l-r+1\}$.  For any $\alpha\in (0,1)$ such that $\alpha n\in \mathbb{N}$, consider the hypergraph
$G$ where $\alpha n $ vertices form a clique and where the remaining $(1-\alpha) n$ vertices form a
$T_{(1-\alpha)n,r+i-2}^{(r)}$ graph (see Figure \ref{fig:Gexample}). Note that $t\left(K_j^{(r)};G\right)= \alpha^j + \frac{(r+i-2)! (1-\alpha)^j}{(r+i-2-j)!(r+i-2)^j}+O\left(\frac1n\right)$ for any $r\leq j\leq r+i-2$ (since cliques of those sizes can be found both in the $\alpha n$-clique and in $T_{(1-\alpha)n,r+i-2}^{(r)}$) and that $t\left(K^{(r)}_j;G\right)=\alpha^j +O\left(\frac1n\right)$ for any $r+i-1\leq j \leq l$ (since cliques of those size can only be found in the $\alpha n$ clique of $G$). For instance, to see the latter, note that that $t\left(K^{(r)}_j;G\right)=\frac{\binom{\alpha n}{j}j!}{n^j}$ for any $r+i-1 \leq j \leq l$.  

\begin{center}
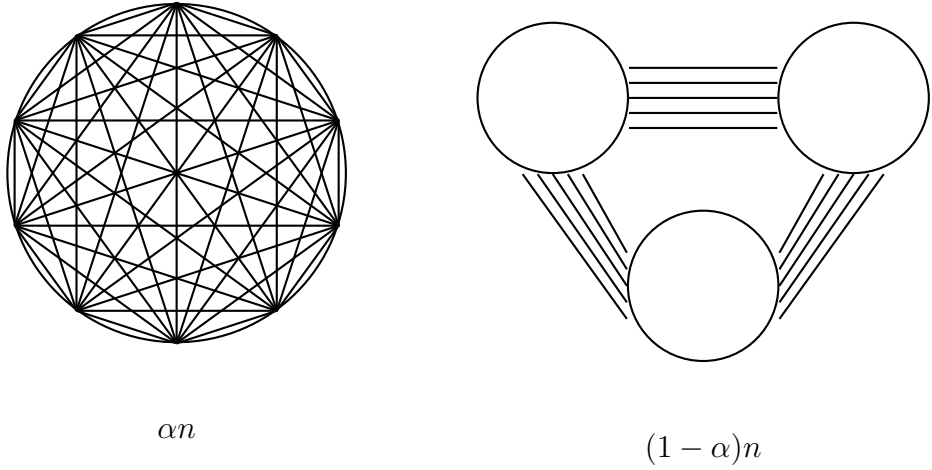
\begin{figure}[ht]
\begin{tikzpicture}[thick]

\path [every node/.style={draw, circle, minimum width=2cm]}]
node (d)[minimum width=4.5cm,label={[label distance=0.15cm]270:$\alpha n$}] at (0,0) {}
  node (a) at (5,1) {}
  node (b) at (9,1) {}
  node [label={[label distance=0.15cm]270:$(1-\alpha)n$}] (c) at (7,-1.5)  {};

\begin{scope}[->,>=latex]
    \foreach \i in {-2,...,2}{%
      \draw[-] ([yshift=\i * 0.2 cm]a.east) -- ([yshift=\i * 0.2 cm]b.west) ;}

    \foreach \i in {-2,...,2}{%
      \draw[-] ([xshift=\i * 0.2 cm]a.south) -- ([yshift=\i * 0.22 cm]c.west) ;}

    \foreach \i in {-2,...,2}{%
      \draw[-] ([yshift=\i * 0.22 cm]c.east) --  ([xshift=-\i * 0.2 cm]b.south) ;}
\end{scope}

\node[
    regular polygon,
    regular polygon sides=10,
    minimum size=4.5cm,
    rotate=180/10,
  ] (a) {};
\foreach \x in {1,...,10}{
\foreach \y in {\x,...,10}{
\draw (a.corner \x) -- (a.corner \y);
};
};
\end{tikzpicture}
\caption{\label{fig:excomplete} $G$ when $i=3$, $r=2$} \label{fig:Gexample}
\end{figure}
\end{center}

Now consider the vector $\mathbf{v}\in \RR^{l+r-1}$ such that $v_h=\left(\frac{\log(t(K^{(r)}_{r+h-1};G))}{\log \alpha}\right)$ for each $1\leq h\leq l-r+1$.As $\alpha\rightarrow 0$ and $n\rightarrow \infty$, we have $v_h\rightarrow 0$ for all $1\leq h\leq i-1$ and $v_h\rightarrow \-(r+h-1)$ for $i\leq h\leq l-r+1$. Since this limit point is exactly $\mathbf{u}_i$ and in $\trop(\G_\U)$, we have the lemma.
\end{proof}
This completes the proof of Theorem~\ref{lem:trop-clique}.
\end{proof}

Note that one can find $\trop(\G_{\U'})$ for $\U'\subset \mathcal{U}$ by projecting down $\trop(\G_\U)$ on the appropriate coordinates. Moreover, a consequence of Theorem~\ref{lem:trop-clique} is that any valid binomial inequality for $\G_\mathcal{U}$ is implied by the
Kruskal-Katona inequalities and $t(K_r^{(r)};G) \leq 1$.

\subsection{Tropicalization of Star Hypergraph Profiles.}\label{sec:example2}

We now give the tropicalization of a collection of generalized stars $\{S^{(r)}(1,c), S^{(r)}(2,c), \ldots, S^{(r)}(l,c)\}$. For any $\mathbf{y} \in \mathbb{R}^{l}$, denote by $y_{S^{(r)}(b,c)}$ the coordinate of $\mathbf{y}$ indexed by $S^{(r)}(b,c)$ for $1 \leq b \leq l$. Also, define $\mathbf{1}_{S^{(r)}(b,c)}$ for some $1\leq b \leq l$ to be the point in $\mathbb{R}^{l}$ with $1$ in the coordinate labeled by $S^{(r)}(b,c)$ and $0$ otherwise.  Finally, let $d_{v_1, v_2, \ldots, v_k}=|\{e\in E(G)|v_1, v_2, \ldots, v_k \in e\}|$ be the \emph{(common) degree of a set of vertices $v_1, \ldots, v_k$}.

\begin{theorem}\label{lem:trop-stars}
Let ${\U}= \{S^{(r)}(1,c), S^{(r)}(2,c), \ldots, S^{(r)}(l,c)\}$, then
$$\trop(\G_\U)=\left\{{\bf y} \in \RR^{l}:\mathbf{a}_b^\top {\bf y} \geq 0, \;\forall \;1\leq b\leq l \right\}$$
where 
$$\begin{array}{ll}
\mathbf{a}_1^\top {\bf y} = -2y_{S^{(r)}(1,c)}+y_{S^{(r)}(2,c)}, & \\
\mathbf{a}_b^\top {\bf y}= y_{S^{(r)}(b-1,c)}-2y_{S^{(r)}(b,c)}+y_{S^{(r)}(b+1,c)} & \textup{  for } 
2\leq b\leq l-1, \textup{  and }\\
\mathbf{a}_l^\top {\bf y}=y_{S^{(r)}(l-1,c)}-y_{S^{(r)}(l,c)} & .
\end{array}$$
Moreover, the extreme rays of $\trop(\G_\U)$ are
$\ubf_1,\ldots, \ubf_l$ where $$\ubf_b=-\sum_{j=1}^{b} j\mathbf{1}_{S^{(r)}(j,c)} -\sum_{j=b+1}^l b\mathbf{1}_{S^{(r)}(j,c)}$$
 for $1\leq b\leq l$.  
\end{theorem}
\begin{proof}
To calculate the homomorphism density of $S^{(r)}(b,c)$ in some graph $G$ with $n$ vertices, we first note that there are $n^{b(r-c)+c}$ maps from $V(S^{(r)}(b,c))$ to $V(G)$. We first decide to which $c$ distinct vertices $v_1, v_2, \ldots, v_c\in V(G)$ to send the center and in which of $c!$ different ways to do so.  Then each of the $b$ edges of $S^{(r)}(b,c)$ can be sent to any of the $d_{v_1, v_2, \ldots, v_c}$ edges containing $v_1, v_2, \ldots, v_c$ in $G$. For each of the $b$ edges, there are $(r-c)!$ different orders to send the vertices in that edge that are not in the center to some chosen edge in $d_{v_1, v_2, \ldots, v_c}$. Thus, for any $b\geq 1$, $1\leq c\leq r-1$ and $r\geq 2$, we have that

\begin{align*}
 t(S^{(r)}(b,c);G) & =   \frac{c!  \sum_{1 \leq v_1 < v_2 < \ldots < v_c \leq n} ((r-c)!d_{v_1,v_2, \ldots, v_c})^b } {n^{b(r-c)+c}} \\
 &  = \frac{c!}{n^c} \sum_{1 \leq v_1 < v_2 < \ldots < v_c \leq n} \left(\frac{(r-c)! d_{v_1, v_2, \ldots, v_c}}{n^{r-c}}\right)^b.
 \end{align*}
Let $\delta_{v_1, v_2, \ldots, v_c}=\frac{(r-c)! d_{v_1, v_2, \ldots, v_c}}{n^{r-c}}$. Consider the uniform measure on $\delta_{1,2,\ldots,c} , \ldots, \delta_{n-c+1,n-c+2, \ldots, n} \in [0,1]$. The $b$th moment of this measure is $$\frac{\sum_{1 \leq v_1 < v_2 < \ldots < v_c \leq n} \delta_{v_1, v_2, \ldots, v_c}^b}{\binom{n}{c}}= t(S^{(r)}(b,c);G) +O\left(\frac1n\right).$$
This connection allows us to write down binomial inequalities that are valid on $(t(S^{(r)}(b,c);G) \,:\, b=1,\ldots, l)$.
They are of the form
$$ m_2 \geq m_1^2, \,\,\, m_{b-1}m_{b+1} \geq m_b^2, \,\,b = 2, \ldots, l-1, \,\,\, m_{l-1} \geq m_{l}.$$
These inequalities follow from H\"{o}lder's inequality \cite[Theorem 18]{MR944909}.
These binomial inequalities imply that the $\mathbf{a}_b^\top {\bf y} \geq 0$ are valid for $\trop(\G_\U)$. Thus if we let $$C=\{{\bf y} \in \RR^{l}:\mathbf{a}_b^\top {\bf y} \geq 0, \;\forall \; 1 \leq b\leq l\},$$ we have $\trop(\G_\U)\subseteq C$. As in Theorem~\ref{lem:trop-clique}, we first characterize the extreme rays of $C$ and then show that they are in $\trop(\G_\U)$ to complete the proof.

\begin{lemma} The extreme rays of the cone $C$ are exactly $\ubf_b$ for $1\leq b\leq l$. 
\end{lemma}

\begin{proof}
Observe that $C=\{{\bf y}\in \RR^{l}: A{\bf y}\geq 0\}$ where
$$ A := \begin{bmatrix}
-2 & 1 & 0 & 0 & 0 & \cdots &0 & 0  \\
1 & -2 & 1 & 0 & 0 & \cdots & 0 & 0  \\
0 & 1 & -2 & 1 & 0 & \cdots & 0 & 0  \\
0 & 0 & 1 & -2 & 1  & \cdots & 0 & 0\\
0 & 0 & 0 & 1 & -2 & \cdots & 0 & 0 \\
\vdots & \vdots & \vdots & \vdots & \vdots & \vdots & \vdots & \vdots \\
0 & 0 & 0 & 0 & 0 & \cdots & 1 & 0 \\
0 & 0 & 0 & 0 & 0 & \cdots & -2& 1 \\
0 & 0 & 0 & 0 & 0 & \cdots & 1 &  -1
\end{bmatrix}.$$
where the $b^{\textup{th}}$ row of $A$ is $\mathbf{a}_b$. Since $A\in \RR^{l \times l}$, the candidate extreme rays of $C$ are obtained by setting $l-1$ of the defining constraints $\mathbf{a}_b^\top {\bf y} \geq 0$ to equality. Since there are exactly $l$ constraints, it implies there are at most $l$ extreme rays. Observe that $\ubf_i=(-1,-2,\ldots, -b, -b,\ldots, -b)$ satisfies all but the $b^{\textup{th}}$ constraint at equality. Since $\ubf_b\in C$, it is an extreme ray for each $1\leq b\leq l$.
\end{proof}

 We will now show that
$C \subseteq  \textup{trop}(\mathcal{G}_\mathcal{U})$ by showing that $\ubf_m \in \trop(\G_\U)$ for $1 \leq m \leq l$. This is done by exhibiting a family of graphs $\{G_n\}$ for each extreme ray $\ubf_m$ for which $$(\log(t(S^{(r)}(b,c);G_n)) \,:\, b=1,\ldots, l)$$ limits to this extreme ray. 

\begin{lemma}
The extreme rays of $C$ are in $\textup{trop}(\mathcal{G}_{\mathcal{U}})$, and hence $C =  \textup{trop}(\mathcal{G}_\mathcal{U})$. \end{lemma}

\begin{proof}

We say $G$ is $k$-regular if $d_{v_1, v_2, \ldots, v_c}=k$ for some $k\in \mathbb{N}$ for every $v_1, v_2, \ldots, v_c\in V(G)$. Moreover, by edge density of a $r$-uniform hypergraph $H$ on $n$ vertices, we will mean the homomorphism density of an edge in $H$ which is $\frac{r!|E(H)|}{n^r}$.

Consider a $r$-uniform hypergraph $G^{(r)}_{n,\rho}$ on $n$ vertices constructed as follows: $(1-\alpha)n$ vertices form a $k$-regular graph with edge density is $\rho$  and the remaining $\alpha n$ vertices form a clique. Call the subgraph formed by the clique $A$, and the subgraph formed by the regular part $B$. Furthermore, any $c$ vertices in $A$  and any $r-c$ vertices in $B$ also form an edge. The parameter $\alpha$ will be chosen later. Also, note that $k=\frac{((1-\alpha)n)^{r-c}\rho}{(r-c)!}+O(n^{r-c-1})$. This can be seen by adding the degrees of all sets of $c$ vertices in $B$, and seeing that each edge gets counted $\binom{r}{c}$ times that way. Thus, 

$$\frac{k \binom{(1-\alpha)n}{c}}{\binom{r}{c}}=|E(B)|=\frac{\rho((1-\alpha)n)^r}{r!},$$ yielding the desired relation between $k$ and $\rho$.

Every set of $c$ distinct vertices $v_1, v_2, \ldots, v_c\in A$ has degree $d_{v_1, v_2, \ldots, v_c} = \binom{n-c}{r-c}$ since that set of vertices forms an edge with any $r-c$ distinct vertices in $G^{(r)}_{n, \rho}$ different from it. Every set of $c$ distinct vertices $v_1, v_2, \ldots, v_c\in B$ has degree $k = \frac{((1-\alpha)n)^{r-c}\rho}{(r-c)!}+O(n^{r-c-1})$ from edges fully in $B$ and, if $r-c \geq c$, then $\binom{(1-\alpha)n-c}{r-2c}\binom{\alpha n}{c}$ gets added to that for edges going between $B$ and $A$. Finally every set of $c$ distinct vertices where $i$ of them come from $A$ and $c-i$ come from $B$ where $\max\{1, 2c-r\} \leq i \leq c-1$ has degree $\binom{\alpha n}{c-i}\binom{((1-\alpha)n}{r-2c+i}$. Thus,

\begin{align*}
t(S^{(r)}&(b,c);G^{(r)}_{n,\rho}) \\
= &\frac{c!}{n^c} \Bigg[ \binom{\alpha n}{c} \left( \frac{(r-c)! \binom{n-c}{r-c}}{n^{r-c}}\right)^b\\
& + \binom{(1-\alpha)n}{c} \left(\frac{(r-c)!}{n^{r-c}} \cdot \left(\frac{((1-\alpha)n)^{r-c}\rho}{(r-c)!}+O(n^{r-c-1}) +\binom{(1-\alpha)n -c}{r-2c}\binom{\alpha n}{c}\right) \right)^b\\
& + \sum_{i=\max\{1,2c-r\}}^{c-1} \binom{\alpha n}{i}\binom{n-\alpha n}{c-i} \left(\frac{(r-c)! \binom{\alpha n}{c-i}\binom{(1-\alpha)n}{r-2c+i}}{n^{r-c}} \right)^b \Bigg],
\end{align*}
where the second and third lines respectively come from sending the center of $S^{(r)}(b,c)$ to $A$ and $B$. The fourth line comes from sending $i$ vertices of the center of $S^{(r)}(b,c)$ to $A$, and the other $c-i$ vertices to the $B$.

As $n\rightarrow \infty$, this goes to

\begin{align*}
\alpha^c + (1-\alpha)^c\left( (1-\alpha)^{r-c} \rho + \frac{(r-c)!}{(r-2c)!}(1-\alpha)^{r-2c}\alpha^c\right)^b \\
+ \sum_{i=\max\{1,2c-r\}}^{c-1} \frac{c!}{i!(c-i)!} \alpha^i (1-\alpha)^{c-i} \left(\frac{(r-c)!\alpha^{c-i} (1-\alpha)^{r-2c+i}}{(c-i)!(r-2c+i)!}\right)^b
\end{align*}

If we choose $\alpha=\rho^{\frac{m}{c}}$, with $\rho<<1$ then the lowest degree part of $t(S^{(r)}(b,c);G^{(r)}_{n,\rho})$ is $\rho^b + \rho^m$. If $m < b$, then $\rho^m$ dominates and if $m \geq b$, then $\rho^b$ dominates.
Thus $\lim_{n\rightarrow \infty, \rho\rightarrow 0} \frac{\log t(S^{(r)}(b,c);G^{(r)}_{n,\rho})}{\log \frac{1}{\rho}}$ equals $-b$ if $b\leq m$ and $-m$ if $b\geq m$. Thus this vector, $(-1,-2,\ldots, -m,-m\ldots, -m)=\ubf_m\in \trop(\G_{\U})$ as claimed.
\end{proof}
This completes the proof of the containment of $C \subseteq \textup{trop}(\mathcal{G}_\mathcal{U})$ and thus Theorem \ref{lem:trop-stars} holds.
\end{proof}

Note that one can find $\trop(\G_{\U'})$ for $\U'\subset \U$ by projecting down $\trop(\G_\U)$ on the appropriate coordinates. Moreover, it follows that any valid binomial inequality for $\G_\mathcal{U}$ is implied by moment inequalities.


\section{Sums of squares profiles and their tropicalizations}\label{sec:SosProfiles}

In this section we introduce {\em sos-profiles} which are semialgebraic sets that
 contain graph profiles. The main result of this section is Theorem~\ref{thm:sos-2by2} which 
 shows that the tropicalization of any sos-profile is a rational polyhedral cone that can be described explicitly. We will use this description in the next section to show 
 strong limitations of sums of squares in recognizing graph density inequalities. 

\subsection{A general framework}
Let $M$ be a symmetric matrix filled with monomials in the finite set of variables $x_1, \ldots, x_s$ for which no $2\times 2$-principal minor is identically zero. For a point
${\bf v} \in \RR^s_{\geq 0}$,
let $M({\bf v})$ denote the matrix obtained by evaluating each entry of $M$ at ${\bf v}$,
and consider the set $$S_M :=  \{ {\bf v} \in \RR^s_{\geq 0} \,:\, M({\bf v}) \succeq 0 \}.$$
Also consider the superset of $S_M$
$$S_M^{2 \times 2} := \{ {\bf v} \in \RR^s_{\geq 0} \,:\, \textup{ all } 2 \times 2 \textup{ principal minors of } M({\bf v}) \textup{ are nonnegative} \}.$$

Both $S_M$ and $S_M^{2 \times 2}$ are (closed) semialgebraic sets in $\RR^s_{\geq 0}$, and they both have the Hadamard property.
Indeed, if ${\bf v}, {\bf w} \in \RR^s_{ \geq 0}$ and $M({\bf v}) \succeq 0$ and $M({\bf w}) \succeq 0$ then
their  Hadamard product which is $M({\bf v} \cdot {\bf w})$ is also positive semidefinite.
Therefore, $S_M$ has the Hadamard property.
If ${\bf v}, {\bf w} \in \RR^s_{ \geq 0} \cap S_M^{2 \times 2}$, then for any $2 \times 2$ principal
minor ${\bf x}^{\bf a}{\bf x}^{\bf b} - {\bf x}^{2\bf c}$ of $M$, we have that
${\bf v}^{\bf {a + b}} \geq  {\bf v}^{2{\bf c}}$ and ${\bf w}^{\bf {a + b}} \geq  {\bf w}^{2{\bf c}}$. Therefore, we also have
${\bf v}^{\bf {a +  b}} {\bf w}^{\bf {a + b}}\geq  {\bf v}^{2{\bf c}} {\bf w}^{2 {\bf c}}$, or equivalently,
$({\bf v} \cdot {\bf w})^{\bf {a+b}} \geq ({\bf v} \cdot {\bf w})^{2 {\bf c}}$. Therefore,
$S_M^{2 \times 2}$ has the Hadamard property.

\begin{lemma}
\begin{enumerate}
\item The set $\log(S_M^{2 \times 2})$ is a polyhedral cone in $\RR^s$.
\item $\textup{trop}(S_M^{2 \times 2}) = \log(S_M^{2 \times 2})$.
\end{enumerate}
\end{lemma}

\begin{proof}
\begin{enumerate}
\item A point $\mathbf{y} \in \log(S_M^{2 \times 2})$ if and only if $\mathbf{y} = \log({\bf v})$ for some ${\bf v} \in S_M^{2 \times 2} \cap \RR^{s}_{>0}$.
A $2 \times 2$ principal minor of $M$ evaluated at ${\bf v} \in \RR^s_{> 0}$  is of the form ${\bf v}^{\bf a} {\bf v}^{\bf b} - {\bf v}^{2{\bf c}}$.
For a ${\bf v} \in \RR^s_{> 0}$,  ${\bf v}^{{\bf a} + {\bf b}} \geq {\bf v}^{2 {\bf c}}$ if and only if $\sum_{i=1}^s (a_i+b_i - 2c_i) \log(v_i) \geq 0$. Thus
$\log(S_M^{2 \times 2})$ is the polyhedral cone in $\RR^s$ defined by the linear inequalities
$\sum_{i=1}^s (a_i+b_i - 2c_i) y_i \geq 0$ obtained from the $2 \times 2$ principal minors of $M$.

\item We already showed that $S_M^{2 \times 2}$ has the Hadamard property.
Since $\log(S_M^{2 \times 2})$ is a polyhedral cone, it coincides with the closure of both its cone hull and convex hull, and so by
Lemma~\ref{lem:tropS}, $\textup{trop}(S_M^{2 \times 2}) = \log(S_M^{2 \times 2})$.
\end{enumerate}
\end{proof}

\begin{corollary} \label{cor:dual cone}
The dual cone, $\textup{trop}(S_M^{2 \times 2})^\ast \subset \RR^s$, is the rational polyhedral cone
generated by the vectors
$$ \left\{ {\bf a} + {\bf b} - 2{\bf c} \,:\, \begin{pmatrix} {\bf x}^{\bf a} & {\bf x}^{\bf c} \\ {\bf x}^{\bf c} & {\bf x}^{\bf b} \end{pmatrix} \textup{ is a } 2 \times 2
\textup{ principal submatrix of } M \right\}.$$
\end{corollary}

\begin{lemma} $\trop(S_M) = \textup{cl}(\textup{conv}(\log(S_M)))$.
\end{lemma}

\begin{proof} The all-ones matrix $M({\bf 1})$ is positive semidefinite and hence ${\bf 1} \in S_M$. The result now follows from $S_M$ having the
Hadamard property and Lemma~\ref{lem:tropS}.
\end{proof}

\begin{theorem} \label{thm:tropSM=tropSM2by2}
Let $M$ be a symmetric matrix filled with monomials in $x_1, \ldots, x_s$ such that no $2\times 2$ minor of $M$ is identically $0$.
Suppose also that $S_M^{2\times 2}$ has interior in $\mathbb{R}^s_{>0}$.  Then $\textup{trop}(S_M) = \textup{trop}(S_M^{2 \times 2})$.
\end{theorem}

\begin{proof}
Since $S_M \subseteq S_M^{2 \times 2}$ we have that
$$\textup{trop}(S_M) = \textup{cl}(\textup{conv}(\log(S_M)) \subseteq \textup{cl}(\textup{conv}(\log(S_M^{2 \times 2}))= \textup{trop}(S_M^{2 \times 2}).$$
To show the reverse containment, we use the following strategy. Pick a ${\bf v} \in \textup{int}(S_M^{2 \times 2})$, or equivalently, $\log({\bf v}) \in \textup{int}(\log(S_M^{2 \times 2})) =  \textup{int}(\textup{trop}(S_M^{2 \times 2}))$. Then show that for a large enough positive integer $k$, $k \log({\bf v})$ lies in $\log(S_M) \subseteq \textup{trop}(S_M)$. Since $\textup{trop}(S_M)$ is a cone,  it must follow that $\log({\bf v}) \in \textup{trop}(S_M)$. Thus we have that
$\textup{int}(\textup{trop}(S_M^{2 \times 2})) \subseteq \textup{trop}(S_M)$ which means that $\textup{trop}(S_M^{2 \times 2}) \subseteq \textup{trop}(S_M)$
since the tropicalizations are closed sets.

Consider $\log({\bf v})$ in the interior of $\log(S_M^{2 \times 2})$. Then ${\bf v}$ lies in the interior of $S_M^{2 \times 2}$ and all
$2 \times 2$ principal minors of $M({\bf v})$ are strictly positive.
Since $\log(S_M^{2 \times 2})$ is a cone,
for any $k > 0$ and integer,
$k \log({\bf v}) = \log({\bf v}^k)$ lies in
$\log(S_M^{2 \times 2})$. We will now argue that if $k$ is large enough then ${\bf v}^k$ is also in
$S_M$ or equivalently, that all principal minors of $M({\bf v}^k)$ are positive. Recall that no $2 \times 2$ principal minor of
$M$ is identically zero.
A $2 \times 2$ principal minor of $M({\bf v}^k)$, namely
$ {\bf v}^{k({\bf a} + {\bf b})} - {\bf v}^{k(2{\bf c})} = ({\bf v}^{{\bf a} + {\bf b}})^k - ({\bf v}^{2{\bf c}})^k$,
is positive since ${\bf v}^{{\bf a} + {\bf b}} > - {\bf v}^{2{\bf c}}$.
Now consider a $l \times l$ principal minor of $M({\bf v}^k)$, and
a term $T$ in the Laplace expansion of its determinant indexed by a non-identity permutation.
Replace every non-diagonal entry in $T$ from position $(i,j), i \neq j$ by the product of the square roots of
the diagonal entries in positions $(i,i)$ and $(j,j)$ in this $l \times l$ principal minor.
Let $T'$ denote the modification of $T$ obtained by replacing all
non-diagonal terms in $T$ as above.
Since all $2 \times 2$ minors of $M({\bf v}^k)$ are positive, we get that
$T' > T$. Note that $T'$ is the product of diagonal entries in the $l \times l$ principal minor we are considering.
Since there are only finitely many terms in the Laplace expansion of the determinant of the $l \times l$ principal minor, we can choose $k$ large enough to ensure that $T'$ is so much bigger than the other terms making the entire determinant positive.

\end{proof}

\begin{example} We now give an example to illustrate the necessity of the condition that no $2 \times 2$ principal minor of
$M$ should be identically zero.
Consider the matrix
$$\begin{pmatrix}x&x^2&x^2\\ x^2&x^3&x^2\\ x^2&x^2&1\end{pmatrix}$$
in which the upper left $2\times 2$ minor is identically $0$. The values of $x$ for which all $2\times 2$ principal minors are
nonnegative is precisely $0\leq x \leq 1$. Therefore, $S_M^{2 \times 2} = [0,1]$ and
$\operatorname{trop}(S_M^{2 \times 2}) = \mathbb{R}_{\leq 0}$.
The determinant of this matrix is $-(x-1)^2x^5$, so the only values that make the matrix positive semidefinite are $x=0$ and $x=1$.
Therefore, $\operatorname{trop}(S_M)$ is the origin which is strictly contained in $\operatorname{trop}(S_M^{2 \times 2})$.
\end{example}

The main take away from what we have so far is that even though $S_M$ might be strictly contained in $S_M^{2 \times 2}$, their
tropicalizations agree and form a polyhedral cone with an explicit inequality description given by the $2 \times 2$ principal minors of $M$.

\subsection{Specialization to graphs}

We now specialize the above results to the case of graphs. For this we begin with a few definitions about the {\em gluing algebra} of graphs,
the reader is referred to Lov\'{a}sz~\cite{LovaszBook} for a broader exposition. A graph is {\em partially labeled} if a subset of its vertices are labeled with elements of $\mathbb{N} := \{1,2,3,\ldots\}$ such that no vertex receives more than one label.
If no vertices of $H$ are labeled then $H$ is {\em unlabeled}.
Let $\A$ denote the vector space of all formal finite $\RR$-linear combinations of partially labeled graphs without isolated vertices, including
the empty graph with no vertices which we denote as $1$. We call an element $a = \sum \alpha_i H_i$ of $\A$ a {\em graph combination}, each $\alpha_i H_i$ a {\em term} of $a$, and each $H_i$ a {\em constituent graph}  of $a$.
Let $\A_\emptyset$ denote the subspace of $\A$ spanned by unlabeled graphs.
We view elements $a \in \A_\emptyset$ as functions that can be evaluated on unlabeled graphs $G$ via homomorphism densities, namely
$t(a;G) = \sum \alpha_i t(H_i;G)$. An element $a = \sum \alpha_i H_i$ of $\A_\emptyset $ is called {\em nonnegative} if $ \sum \alpha_i t(H_i;G)\geq 0$ for all graphs $G$.

The vector space $\A$ has a product defined as follows. For two labeled graphs   $H_1$ and $H_2$, form the new labeled graph $H_1H_2$  by gluing together the vertices in the two graphs with the same label, and keeping only one copy of any edge that may have doubled in the process.
Equipped with this product, $\A$ becomes an $\RR$-algebra with the empty graph $1$ as its multiplicative identity.
The algebra $\A$ admits a simple linear map into $\A_\emptyset$ that removes the labels in a graph combination to create a graph combination of unlabeled graphs. We call this map {\em unlabeling} and denote it by $[[ \cdot ]]$.
A {\em sum of squares (sos)} in $\A_\emptyset$ is a finite sum of unlabeled squares of graph combinations $a_i \in \A$, namely,
$\sum [[ a_i^2 ]]$. A sum of squares is a nonnegative graph combination.

A $d$-{\em sos graph combination}
is an sos $a= \sum [[a_j^2]]$ where each constituent graph in $a_j$ is partially labeled and has at most $d$ edges. This means that every constituent graph of  $a$ has at most $2d$ edges.
For a fixed $d \in \mathbb{N}$, it follows from results of \cite{MR3856524} (see also \cite{Potechin}) that any $d$-sos graph combination can be written using only finitely many, say $\ell(d)$, labels.
Let $\mathcal{B}_d$ denote the set containing the empty graph $1$ with no vertices, and
all partially labeled graphs with labels $1,\dots, \ell(d)$, and at most $d$ edges and no isolated vertices.
Define
$$\mathcal{V}_d = \{ [[ab]] \textup{ connected  } \,:\, a,b \in \mathcal{B}_d \} \backslash \{1\}.$$
A $d$-sos graph combination is a sum of squares of graph combinations in the span of $\mathcal{B}_d$. i.e.,
if $a = \sum [[a_j^2]]$ is $d$-sos then $a_j \in \textup{span}(\mathcal{B}_d)$.
Also, any term in a $d$-sos graph combination $a$ is a monomial in the elements of $\mathcal{V}_d$ (including constant terms)
and each constituent graph in $a$ has at most $2d$ edges.

\begin{definition}  \begin{enumerate}
\item The $d$-{\em sos-profile}, denoted $\mathcal{S}_d$, is the set of all
${\bf v} \in \RR^s_{\geq 0}$ such that $a({\bf v}) \geq 0$ for all $d$-sos graph combinations $a$ and where $s=|\mathcal{V}_d|$.
\item The $(\mathcal{U},d)$-{\em sos-profile} denoted as $\mathcal{S}_{\mathcal{U},d}$ is the projection of $\mathcal{S}_d$ on coordinates corresponding to graphs in $\mathcal{U}$.
\end{enumerate}
\end{definition}

Let $\mathcal{M}_d$ be the {\em moment matrix} of size $|\mathcal{B}_d| \times |\mathcal{B}_d|$ which is defined as the matrix
with rows and columns indexed by
the graphs in $\mathcal{B}_d$ and whose $(A,B)$-entry is  $[[AB]]$. Such a matrix is called a {\em connection matrix} in \cite{LovaszBook}.
Every entry in $\mathcal{M}_d$ is a monomial in the elements of $\mathcal{V}_d$ (including $1$)
and the corresponding graph has at most $2d$ edges. Further, the entries of $\mathcal{M}_d$ and the monomials that appear in $d$-sos graph combinations are the same.

\begin{lemma}  \label{lem:sos-profile is SMd}
The $d$-sos-profile $\mathcal{S}_d = \{ {\bf v} \in \RR^s_{ \geq 0} \,:\, \mathcal{M}_d({\bf v}) \succeq 0 \} = S_{\mathcal{M}_d}$.
\end{lemma}

\begin{proof}
For any $Q \succeq 0$, $\langle \mathcal{M}_d, Q \rangle$ is a $d$-sos graph combination $\sum_j [[a_j^2]]$
since $a_j \in \textup{span}(\mathcal{B}_d)$.
Conversely, any $d$-sos graph combination can be written as $\langle \mathcal{M}_d, Q \rangle$ for some $Q \succeq 0$.
Therefore, ${\bf v} \in \RR^s_{\geq 0}$ lies in $\mathcal{S}_d$ if and only if $ \langle \mathcal{M}_d({\bf v}), Q \rangle \geq 0$ for all
$Q \succeq 0$ which happens if and only if $\mathcal{M}_d({\bf v}) \succeq 0$.
\end{proof}

The above lemma shows that the sos-profile $\mathcal{S}_d$ is a
semialgebraic set in $\RR^s_{\geq 0}$ since the condition $\mathcal{M}_d({\bf v}) \succeq 0$ is equivalent to the (finitely many) principal minors of
$\mathcal{M}_d({\bf v})$ being nonnegative, and each such minor is a polynomial in the elements of $\mathcal{V}_d$.

We now note the connection between the graph profile $\mathcal{G}_{\mathcal{V}_d}$ and the $d$-sos-profile $\mathcal{S}_d$.
Recall that  $\mathcal{G}_{\mathcal{V}_d}$ is the closure of all points $(t(C;G)  \,:\, C \in \mathcal{V}_d)$ as $G$ varies over all unlabeled graphs.
Tautologically, $\mathcal{G}_{\mathcal{V}_d}$ is also the set of all points in $\RR^s_{\geq 0}$ on which all nonnegative graph combinations in
the polynomial ring $\RR[\mathcal{V}_d]$ are nonnegative. Since there is no bound to the number of edges in the constituent graphs of such nonnegative graph combinations, the profile
$\mathcal{G}_{\mathcal{V}_d}$ may not be semialgebraic, and as mentioned in the introduction, the profile of edge and triangle is not semialgebraic. A $d$-sos graph combination is also a nonnegative graph combination in
$\RR[\mathcal{V}_d]$, but one in which the constituent graphs cannot have more than $2d$ edges. Therefore,
we immediately get that the graph profile $\mathcal{G}_{\mathcal{V}_d}$ is contained in $\mathcal{S}_d$.
Indeed, if ${\bf v}\in \mathbb{R}^s_{\geq 0}\cap \mathcal{G}_{\mathcal{V}_d}$, then ${\bf v}{\bf v}^\top = \mathcal{M}_d({\bf v}) \succeq 0$. If we are given a specific set of finite connected graphs $\mathcal{U}$, then $\mathcal{U} \subseteq \mathcal{V}_d$ for $d$ large enough. Certainly, any finite connected graph $H\in \mathcal{U}$ can be written as the symmetrized product of two partially labeled finite graphs, namely the empty graph and an unlabeled copy of $H$ itself, which is in $\mathcal{V}_{d}$ for any $d\geq |E(H)|$. We then have that $\mathcal{G}_\mathcal{U}$ is a projection of the graph profile $\mathcal{G}_{\mathcal{V}_d}$, and contained in  the
$(\mathcal{U},d)$-sos-profile $\mathcal{S}_{\mathcal{U},d}$.

\begin{example}\label{ex:Sd}
Consider the graph profile $\mathcal{G}_\mathcal{U}$ for $\mathcal{U}=\{\uvedge, \uHthree \}$ from Figure \ref{fig:edge-triangle-profile1}. We will show that $(0.7, 0.12)$ is not in $\mathcal{G}_\mathcal{U}$. We know that $\mathcal{G}_\mathcal{U}$ is contained in $\mathcal{S}_{\mathcal{U},d}$ for all $d \geq 2$.

From Lemma \ref{lem:sos-profile is SMd}, $\mathcal{S}_d$ contains all the points $\bf{v}$ that make $\mathcal{M}_d({\bf v})\succeq 0$. To show that $(0.7, 0.12)\not\in \mathcal{G}_\mathcal{U}$, we can instead prove $(0.7,0.12) \not\in \mathcal{S}_{\mathcal{U},2}$ by observing that any point ${\bf v}\in \mathbb{R}^s$ where the $\uvedge$ component is $0.7$ and the $\uHthree$ component is $0.12$ is such that $\mathcal{M}_2({\bf v})\not \succeq 0$. Indeed, consider the principal submatrix of $\mathcal{M}_2$ corresponding to the graphs
$$\vedge{1}{}, \vedge{2}{}, \vedge{1}{2}, \Htwo{1}{2}{3}, \Htwo{2}{1}{3} \in \mathcal{B}_2:$$

$$\begin{pmatrix}
\uHtwo & \uvedge\uvedge & \uHtwo & \uthreestar & \upthree\\
\uvedge\uvedge & \uHtwo & \uHtwo & \upthree & \uthreestar\\
\uHtwo& \uHtwo & \uvedge & \uHtwo & \uHtwo \\
\uthreestar & \upthree & \uHtwo & \uHtwo & \uHthree \\
\upthree & \uthreestar & \uHtwo & \uHthree & \uHtwo
\end{pmatrix}.$$

The bottom $3\times 3$ principal minor and the top $2\times 2$ principal minor,

$$\det\begin{pmatrix}
\uvedge & \uHtwo & \uHtwo \\
\uHtwo & \uHtwo & \uHthree \\
\uHtwo & \uHthree & \uHtwo
\end{pmatrix} \quad \textup{and } \quad \det\begin{pmatrix}
\uHtwo & \uvedge\uvedge\\
\uvedge\uvedge & \uHtwo
\end{pmatrix},$$
yield that $-2\uHtwo^3+0.94\uHtwo^2-0.01008\geq 0$  and $\uHtwo^2 -0.7^4 \geq 0$ when $\uvedge=0.7$ and $\uHthree=0.12$. However, there is no value of $\uHtwo$ in $[0,1]$ that satisfy both inequalities. Thus the point $(0.7,0.12)\not\in \mathcal{S}_{\mathcal{U},2}$.

Alternatively, to show that $(0.7, 0.12)\not\in \mathcal{G}_\mathcal{U}$, we could have shown that $(0.7,0.12) \not\in \mathcal{S}_{\mathcal{U},3}$ by thinking about $\mathcal{S}_3$ as the set of points that evaluate nonnegatively on all $3$-sos graph combinations. As seen in \cite{LovaszBook}, the Goodman bound can be written as a $3$-sos combination:
 $$ [[(\vedge{2}{3}-\Htwo{2}{1}{3}-\Htwo{3}{1}{2}+\Hthree{1}{2}{3})^2]]+[[(\vedge{1}{}-\vedge{2}{})^2]]=\uHthree-2\uvedge\uvedge+\uvedge,$$ so $\uHthree-2\uvedge\uvedge+\uvedge \geq 0$ is a valid inequality for $\mathcal{S}_3$ and $\mathcal{S}_{\mathcal{U},3}$. Since $0.12-2\cdot 0.7^2+0.7<0$, $(0.7, 0.12)\not\in \mathcal{S}_{\mathcal{U},3}$.

Finally, we note that $\mathcal{S}_{\mathcal{U},3}$ strictly contains $\mathcal{G}_\mathcal{U}$ in this case, that is, there are points $\bf{v}$ such that $\mathcal{M}_3({\bf v})\succeq 0$, but such that the projection of ${\bf v}$ on the coordinates corresponding to graphs in $\mathcal{G}_\mathcal{U}$ is not in $\mathcal{G}_{\mathcal{U}}$. It was shown in \cite{Potechin} that one needs sums of squares of arbitrarily high degree to carve out $\mathcal{G}_\mathcal{U}$ when $\mathcal{U}=\{\uvedge,\uHthree\}$.
\end{example}

\medskip

Lemma~\ref{lem:sos-profile is SMd} suggests that we may be able to apply Theorem~\ref{thm:tropSM=tropSM2by2} to $\mathcal{S}_d$.  The first hurdle
is that some principal minors of $\mathcal{M}_d$ might be identically zero. For example, the principal minor with rows (columns)
indexed by $\uvedge$ and $\uHtwo$, is $\uvedge^2 \uHtwo^2 - \uvedge \ \uHtwo \uvedge\ \uHtwo = 0$.

Say that two partially labeled graphs are isomorphic if they are isomorphic as labeled graphs.
In particular, two isomorphic labeled graphs have the same labels.

\begin{lemma}
A $2\times 2$-principal minor in $\mathcal{M}_d$ is identically zero if and only if the corresponding rows (columns) correspond to two graphs
with isomorphic labeled components.
\end{lemma}

\begin{proof}
Consider a $2\times 2$-principal minor in $\mathcal{M}_d$ corresponding to the rows and columns indexed by $H_1$ and $H_2$ where the labels of $H_1$ and $H_2$ (if any) are contained in some finite set $L$. The minor is thus equal to $[[H_1^2]][[H_2^2]] - [[H_1 H_2]] [[H_1 H_2]]$. We now show that this expression is a sum of squares. Let $\tilde{H}_i$ be the same graph as $H_i$ for $i\in \{1, 2\}$ but where any label $l$ becomes $l+|L|$. Note that the label sets of $H_i$ and $\tilde{H}_j$ do not intersect for $i,j\in \{1,2\}$, that $[[H_i^2]]=[[\tilde{H}_i^2]]$ for $i\in \{1,2\}$ and $[[H_1H_2]]=[[\tilde{H}_1\tilde{H}_2]]$. Thus the minor is equal to

\begin{align*}
\frac{1}{2} [[(H_1\tilde{H}_2 - \tilde{H}_1H_2)^2]] & = \frac{1}{2}[[ H_1^2\tilde{H}_2^2 - 2 H_1H_2\tilde{H}_1\tilde{H}_2+\tilde{H}_1^2H_2^2]]\\
& = \frac{1}{2}\left([[H_1^2]][[\tilde{H}_2^2]] - 2[[H_1 H_2]][[\tilde{H}_1 \tilde{H}_2]] + [[\tilde{H}_1^2]] [[H_2^2]] \right) \\
& = \frac{1}{2} \left([[H_1^2]][[H_2^2]] - 2[[H_1H_2]][[H_1H_2]] + [[H_1^2]][[H_2^2]] \right)
\end{align*}
as desired. Indeed, to go from the first to the second line, one simply needs to note that the symmetrization of products of graphs that have no labels in common is equal to the product of the symmetrization of those graphs.

Thus, if the minor is identically zero, $[[(H_1\tilde{H}_2 - \tilde{H}_1H_2)^2]] =0$. From Lemma 2.3 of \cite{BRST18}, the only way this can be so is if $H_1 \tilde{H}_2=\tilde{H}_1H_2$. Let $H_1=H_1^l H_1^u$ and $H_2=H_2^l H_2^u$ where $H_i^l$ is the graph $H_i$ restricted to components that contain at least one label, and $H_i^u$ is the graph $H_i$ restricted to components that are unlabeled. Then $\tilde{H}_i=\tilde{H}_i^l H_i^u$. Thus, $H_1\tilde{H}_2=\tilde{H}_1H_2$ is equivalent to $H_1^lH_1^u \tilde{H}_2^lH_2^u=\tilde{H}_1^lH_1^uH_2^lH_2^u$ which is equivalent to $H_1^l \tilde{H}_2^l=\tilde{H}_1^lH_2^l$.

Suppose $H_1$ and $H_2$ are two graphs where the labeled components are isomorphic, i.e., $H_1^l=H_2^l$. Note that this includes the case when both $H_1$ and $H_2$ are fully unlabeled. Observe that $\tilde{H}_1^l=\tilde{H}_2^l$. Thus, $H_1^l \tilde{H}_2^l=\tilde{H}_1^lH_2^l$ and the corresponding $2\times 2$-minor is identically zero.

Suppose now that the labeled components of $H_1$ and $H_2$ are not isomorphic, i.e., $H_1^l\neq H_2^l$ and $\tilde{H}_1^l \neq \tilde{H}_2^l$. (Note that this implies that at least one of those two graphs contain a label, otherwise, their labeled parts would be the same.) Moreover, note that $\tilde{H}_1^l \neq H_2^l$ and $H_1^l \neq \tilde{H}_2^l$ as the labels come from non-intersecting label sets. Thus it is impossible that $H_1^l\tilde{H}_2^l=\tilde{H}_1^lH_2^l$, and the corresponding $2\times 2$-minor cannot be identically zero.

\end{proof}

To avoid $2 \times 2$ principal minors that are identically $0$ in $\mathcal{M}_d$,
we restrict $\mathcal{B}_d$ to $\tilde{\mathcal{B}}_d$, the subset  containing the empty graph $1$, and
all partially labeled graphs without unlabeled connected components and isolated vertices.
Every graph in $\tilde{\mathcal{B}}_d$ still has at most $d$ edges.
Furthermore, let $\tilde{\mathcal{M}}_d$ be the moment matrix for $\tilde{\mathcal{B}}_d$.
Then the next corollary follows from the previous lemma.

\begin{corollary}
No $2\times 2$-principal minor in the symmetric matrix $\tilde{\mathcal{M}}_d$ is identically zero.
\end{corollary}

We now replace $\mathcal{M}_d$ with $\tilde{\mathcal{M}}_d$ in Lemma~\ref{lem:sos-profile is SMd}.

\begin{lemma} The $d$-sos-profile $\mathcal{S}_d$ coincides with $S_{\tilde{\mathcal{M}}_d}=
\{ {\bf v}  \in \RR^s_{ \geq 0} \,:\, \tilde{\mathcal{M}}_d({\bf v}) \succeq 0 \}$.
\end{lemma}

\begin{proof}
By Lemma~\ref{lem:sos-profile is SMd} it suffices to argue that for ${\bf v} \geq 0$,
$\tilde{\mathcal{M}}_d({\bf v}) \succeq 0$ if and only if $\mathcal{M}_d({\bf v}) \succeq 0$.
Since $\tilde{\mathcal{M}}_d$ is a principal submatrix of $\mathcal{M}_d$,
$\mathcal{M}_d({\bf v}) \succeq 0$ implies $\tilde{\mathcal{M}}_d({\bf v}) \succeq 0$.

Consider a
graph $F = F^u F^l$ where $F^u$ is unlabeled and every component in $F^l$ has at least one label.
Then the row of  $\mathcal{M}_d$ indexed by $F$ is $(F^u [[F^l H]] \,:\, H \in \mathcal{B}_d)$.
Therefore the corresponding row of $\mathcal{M}_d({\bf v})$ is $(F^u({\bf v})[[F^l H]]({\bf v}) \,:\, H \in \mathcal{B}_d)$.
This means that every term in a principal minor of $\mathcal{M}_d({\bf v})$ (expanded in terms of permutations) involving the row indexed by
$F$ contains the common factor
$F^u({\bf v})$ which is nonnegative. Factoring this out, we obtain a principal minor of $\tilde{\mathcal{M}}_d({\bf v})$. Thus
if $\tilde{\mathcal{M}}_d({\bf v}) \succeq 0$ then $\mathcal{M}_d({\bf v}) \succeq 0$.
\end{proof}

The second requirement in Theorem~\ref{thm:tropSM=tropSM2by2} is that $S_{\tilde{\mathcal{M}}_d^{2 \times 2}}$ has an interior.
By Lemma~\ref{lem:sos-profile is SMd}, $\mathcal{G}_{\mathcal{V}_d} \subseteq
\mathcal{S}_d \subseteq  S_{\tilde{\mathcal{M}}_d^{2 \times 2}}$ and since $\mathcal{G}_{\mathcal{V}_d}$ has an interior \cite{MR538044}
so do $\mathcal{S}_d$ and $S_{\tilde{\mathcal{M}}_d^{2 \times 2}}$. We can now apply Theorem~\ref{thm:tropSM=tropSM2by2}.

\begin{theorem}\label{thm:sos-2by2} The $d$-sos-profile $\mathcal{S}_d$ is a basic semialgebraic set in $\RR^s_{\geq 0}$ containing the graph profile
$\mathcal{G}_{\mathcal{V}_d}$, and its tropicalization, $\textup{trop}(\mathcal{S}_d)$, is the rational polyhedral cone
$\log(S_{\tilde{\mathcal{M}}_d}^{2 \times 2})$.
\end{theorem}

\begin{corollary}
The $(\mathcal{U},d)$-sos-profile $\mathcal{S}_{\mathcal{U},d}$ is a semialgebraic set containing the graph profile $\mathcal{G}_\mathcal{U}$ for all $d$. Its tropicalization, $\textup{trop}(\mathcal{S}_{\mathcal{U},d})$, is the projection of the rational polyhedral cone
$\log(S_{\tilde{\mathcal{M}}_d}^{2 \times 2})$ onto the coordinates indexed by $\mathcal{U}$. In particular, $\textup{trop}(\mathcal{S}_{\mathcal{U},d})$ is also a rational polyhedral cone.
\end{corollary}

\begin{example}
Let $d = 1$. Then $\mathcal{V}_1 = \{\uvedge, \uHtwo \}$ is the set of all unlabeled connected graphs that can be obtained as a product of two partially labeled graphs with at most one edge. The $1$-sos-profile
$\mathcal{S}_1$,  is the set of all ${\bf v } \in \RR^2_{\geq 0}$ that evaluate nonnegatively on all $1$-sos polynomials in $\RR[\uvedge, \uHtwo]$.
The graph profile $\mathcal{G}_{\mathcal{V}_1}$ is shown in Figure~\ref{fig:edge-path-profile}. The lower bound consists of $dn$-regular graphs with $d \in [0,1]$. The upper bound on the right consists of a clique on $\alpha n$ vertices for some $\alpha\in [\frac{1}{2},1]$, and the upper bound on the left consists of the complement of such graphs \cite{MR505076}.

\begin{figure}[ht]
  \centering
\begin{tikzpicture}[scale=0.7]
\begin{axis}[xlabel=\large{Edge density \uvedge}, ylabel=\large{Two-path density \uHtwo}, xmin=0, ymin=0, xmax=1, ymax=1, xtick={0,1}, ytick={0,1}, samples=50]
\addplot[fill=gray!50, draw=none, domain=1/2:1] {x^(3/2)} \closedcycle;
\addplot[thick, samples=50,domain=1/2:1] {x^(3/2)};
\addplot[fill=gray!50, draw=none, domain=0:1-2^(1/2)/2] ({2*x-x^2},{-x^3+x^2+x}) \closedcycle;
\addplot [thick, domain=0:1-2^(1/2)/2,samples=50]({2*x-x^2},{-x^3+x^2+x});
\addplot [dashed, domain=1-2^(1/2)/2:1,samples=50]({2*x-x^2},{-x^3+x^2+x});
\addplot[dashed, samples=50,domain=0:1/2] {x^(3/2)};
\addplot[fill=white, draw=none, domain=0:1] {x^2} \closedcycle;
\addplot[thick, samples=30,domain=0:1] {x^2};
\end{axis}
\end{tikzpicture}
\caption{\label{fig:edge-path-profile} The graph profile of $\mathcal{V}_1$}
\end{figure}
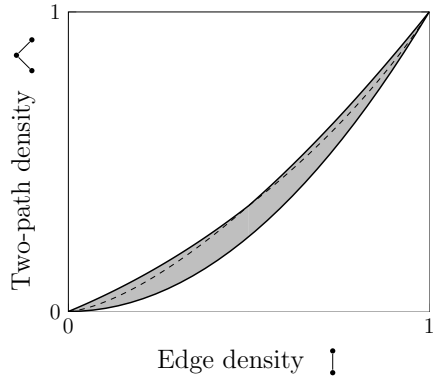

We will show that
$\textup{trop}(\mathcal{G}_{\mathcal{V}_1})=\textup{trop}(\mathcal{S}_1)$.

Consider $\tilde{\mathcal{B}}_1=\{1, \vedge{1}{}, \vedge{2}{}, \vedge{1}{2}\}$.
Then

$$\tilde{\mathcal{M}}_1=\begin{pmatrix}
1 & \uvedge & \uvedge & \uvedge\\[6pt]
\uvedge & \uHtwo & \uvedge\uvedge & \uHtwo\\[6pt]
\uvedge & \uvedge\uvedge & \uHtwo & \uHtwo\\[6pt]
\uvedge & \uHtwo & \uHtwo & \uvedge
\end{pmatrix}$$

After removing redundancies, the six $2 \times 2$ principal minors of $\tilde{\mathcal{M}}_1$
yield the following two inequalities:
$$ \uHtwo-\uvedge^2\geq 0 \,\,\,\textup{ and } \,\,\, \uvedge - \uHtwo \geq 0.$$
Therefore, $\log(S_{\tilde{\mathcal{M}}_1^{2\times 2}}) = \textup{trop}(\mathcal{S}_1)$ is the cone generated by
the rays $(-1,-1)$ and $(-1,-2)$. As we saw in the star example of Section \ref{sec:example2}, this cone coincides with $\textup{trop}(\mathcal{G}_{\mathcal{V}_1})$ when
$\mathcal{V}_1=\{\uvedge, \uHtwo\}$.  \end{example}
\bigskip

\subsection{Sos-testable graph combinations}

We now introduce the notion of an sos-testable graph combination which plays an important role in the next section.

\begin{definition}
A graph combination $a$ is {\em sos-testable} if $a \geq 0$ on a $d$-sos-profile $\mathcal{S}_d$ for some $d$.
\end{definition}

\begin{theorem}\label{thm:mult}
Let $a,b,c$ be graph combinations such that $b\neq 0$ and $c$ are sos-testable. If $ab=c$ then $a$ is sos-testable.
\end{theorem}

\begin{proof}
Let $\mathcal{U}=\{F_1,\ldots, F_k\}$ be the set of connected components of graphs in $a$, $b$ and $c$. There exist $d_1,d_2 \in \mathbb{N}$ such that $b$ is nonnegative on $\mathcal{S}_{\mathcal{U},d_1}$ and $c$ is nonnegative on $\mathcal{S}_{\mathcal{U},d_2}$. Let $d=\max \{d_1,d_2\}$.
We will prove that $a \geq 0$ on $\mathcal{S}_{\mathcal{U},d}$ making it sos-testable.

We first argue that every neighborhood of $\mathbf{1}$ has a ball contained in $\G_{\U}$, i.e., for every $r>0$ there exists a $\tilde{r}>0$,  and
$\mathbf{w} \in \RR^{|\U|}$ such that $B(\mathbf{w},\tilde{r})\subseteq \G_{\U}\cap B(\mathbf{1},{r})$. Here $B(\mathbf{w}, \tilde{r})$ denotes the closed ball of radius
$\tilde{r}$ around $\mathbf{w}$. From Theorem 1~\cite{erdHos1979strong}, we have that there exists ${\mathbf{z}}\in \RR^{|\U|}$ and $\epsilon>0$ such that $B({\mathbf{z}},\epsilon)\subseteq \G_{\U}$. Thus, for every $\mathbf{y} \in B({\mathbf{z}},\epsilon)$ there exists a sequence of graphs $G_1,\ldots, G_n,\ldots$, where $|V(G_n)| = n$,  such that $\lim_{n\rightarrow \infty} (t(F_1;G_n), \ldots, t(F_k;G_n))=\mathbf{y}$.
Now fix $r>0$. Then consider the graph sequence $H_n$ which consists of $G_{\delta n}$ along with a disjoint copy of $K_{(1-\delta)n}$ for each $n$, where $\delta >0$ will be fixed later and we ignore integrality issues with $\delta n$. For an $F_i$, we have $t(F_i,H_n)=(1-\delta)^{|V(F_i)|}+\delta^{|V(F_i)|} t(F_i,G_{\delta n})$. Thus $\lim_{n\rightarrow \infty} t(F_i,H_n)=(1-\delta)^{|V(F_i)|}+\delta^{|V(F_i)|} y_i$. Let
$\phi:\RR^{|\U|}\rightarrow \RR^{|\U|}$ denote the map where
$$\phi(\mathbf{y})=\left( (1-\delta)^{|V(F_1)|}+\delta^{|V(F_1)|} y_1, \ldots, (1-\delta)^{|V(F_k)|}+\delta^{|V(F_k)|} y_k\right).$$
By construction, $\phi(\mathbf{y}) \in \G_\U$.
If we set $\delta =\frac{r}{\max_i{|V(F_i)|}}$ then $(1-\delta)^{|V(F_i)|}\geq 1-\delta {|V(F_i)|} \geq 1-r$. Thus $\phi(\mathbf{y})\geq 1-r$ and $\phi(\mathbf{y})\in B(\mathbf{1},r)$. Therefore, we have $\phi(B({\mathbf{z}},\epsilon))\subseteq \G_\U \cap B(\mathbf{1},r)$.
Moreover, the Jacobian of $\phi$ is just the diagonal matrix with $i^{th}$ diagonal entry  $\delta^{|V(F_i)|}$ and thus has a non-zero determinant. Thus $\phi(B({\mathbf{z}},\epsilon))$ is full-dimensional and contains a ball $B(\mathbf{w},\tilde{r})$ for some $\mathbf{w}
\in \RR^{\U},\tilde{r}>0$. Therefore, we have $B(\mathbf{w},\tilde{r})\subseteq \G_{\U}\cap B(\mathbf{1},{r})$ as claimed.

Now suppose there exists ${\bf x} \in \mS_{\mathcal{U},d}$ such that $a(\mathbf{x})<0$.
Since $\mS_{\mathcal{U},d}$ has the Hadamard property and $\G_{\mathcal{U}}\subset \mS_{\mathcal{U},d}$, we see that any neighborhood of $\mathbf{x}$ in $\mS_{\mathcal{U},d}$ also contains a closed ball by applying the Hadamard property to $\mathbf{x}$ and the closed ball in the neighborhood of $\mathbf{1}$. 
Since $a(\mathbf{x})<0$ and $a$ is polynomial function in the coordinates indexed by $\mathcal{U}$, and therefore continuous, it follows that there exists $\tilde{\mathbf{x}}\in \mS_{\mathcal{U},d}$ such that $a(\tilde{\mathbf{x}})<0$ and a closed ball ${B}$ around $\tilde{\mathbf{x}}$ is contained in $\mS_{\mathcal{U},d}$. Since $b$ and $c$ are sos-testable, we have $b(\mathbf{u}),c(\mathbf{u})\geq 0$ for all $\mathbf{u}\in B$. Moreover, since $b\neq 0$, there exists $\hat{\mathbf{x}} \in{B}$ such that $a(\hat{\mathbf{x}})<0$ and $b(\hat{\mathbf{x}})>0$. Then we have $ab(\hat{\mathbf{x}})<0$ while $c(\hat{\mathbf{x}})\geq0$ which is a contradiction.
\end{proof}

\begin{corollary} \label{cor:not sostestable implies not rational sos}
If a graph combination $a$ is not sos-testable, it is not a rational sos.
\end{corollary}

\begin{proof}
For the sake of contradiction, suppose $a$ is a rational sos, i.e., $a=\frac{c}{b}$ where $b\neq 0$ and $b$ and $c$ are sos. Then we have $ab=c$. Moreover, $b$ and $c$ are sos-testable since they are sos. Thus $a$ must also be sos-testable from Theorem~\ref{thm:mult} which is a contradiction.
\end{proof}


\section{Limitations of Sums of Squares}\label{sec:limitations}

In this section we will use the $d$-sos-profile $\mathcal{S}_d$ and the $(\mathcal{U},d)$-sos-profile $\mathcal{S}_{\mathcal{U},d}$
defined in Section~\ref{sec:SosProfiles} to show that there are simple binomial graph density inequalities that 
are not sos-testable. This means that sums of squares do not recognize these inequalities. 

Following \cite{BRST18}, we call an unlabeled graph $H$, a \emph{trivial square}, if whenever $H=[[F^2]]$ then $F$ must be a fully
labeled copy of $H$. For example, $\upthree$ is a trivial square.

\begin{theorem} \label{thm:noksos}
Let $\oversl{H}$ and $\undersl{H}$ be two graphs with the same number of edges where the former is a trivial square in which
every vertex has degree $p$ or $p+1$ for some integer $p\geq 1$, and the degree of any vertex in $\undersl{H}$ is at most $p+1$. Then for any $k\geq 2|E(\undersl{H})|+1$, the inequality $\oversl{H}^k-\undersl{H}^{k+1}\geq 0$ is not sos-testable.
In particular, $\upthree^k-\uvedge^{3(k+1)}$ is not sos-testable for $k \geq 7$.
\end{theorem}

One can also show that $\upthree^k-\uvedge^{3k+1}$ is not sos-testable, for $k$ big enough, by using a similar strategy as the one described below. Theorem~\ref{thm:noksos} has the following corollary.

\begin{corollary}\label{cor:1}
 Let $\oversl{H}$ and $\undersl{H}$ be two graphs with the same number of edges where the former is a trivial square in which  every vertex has degree $p$ or $p+1$ for some integer $p\geq 1$, and the degree of any vertex in $\undersl{H}$ is at most $p+1$. Then $\oversl{H}-\undersl{H}$ is not sos-testable and cannot be written as a rational sos. In particular, $\upthree-\uvedge^3$ is not sos-testable and cannot be written as a rational sos.
 \end{corollary}
\begin{proof}
Observe that if $\oversl{H}-\undersl{H}$ is sos-testable then $\oversl{H}^k-\undersl{H}^k$ is sos-testable for every integer $k\geq 1$. Furthermore, $\oversl{H}^k-\undersl{H}^{k+1}$ is also sos-testable for every integer $k\geq 1$ since graph densities lie in $[0,1]$. Therefore, if $\oversl{H}$ and $\undersl{H}$ satisfy the conditions of Theorem~\ref{thm:noksos}, $\oversl{H}-\undersl{H}$ is not sos-testable.
By Corollary~\ref{cor:not sostestable implies not rational sos} we then have that $\oversl{H}-\undersl{H}$  is not a rational sos.
For the last claim, observe that $\oversl{H}=\upthree$ and $\undersl{H}=\uvedge^3$ satisfies the conditions of Theorem~\ref{thm:noksos}.
\end{proof}

Before we prove Theorem~\ref{thm:noksos} we make a few definitions.
Recall that every term in a $d$-sos graph combination is a constant times a monomial in the elements of $\mathcal{V}_d$. Therefore, the coordinates of vectors in both $\mathcal{S}_d$ and $\trop(\mathcal{S}_d)$ are indexed by graphs in $\mathcal{V}_d$, and $|\mathcal{V}_d|=s$. In what follows we will
assume that we have fixed an ordering of the elements of $\mathcal{V}_d$.
For any $\mathbf{x} \in \RR^s$ and $H \in \mathcal{V}_d$, denote by $x_H$ the coordinate
of $\mathbf{x}$ indexed by $H$. Also, 
define $\mathbf{1}_H$ to be
the point in tropical space with $1$ in the coordinate labeled by $H$ and $0$ otherwise, the {\em indicator vector} of $H$.

\begin{definition}
Let $G$ be an unlabeled graph with factorization $G=C_1^{\alpha_1}\cdots C_s^{\alpha_s}$ into connected unlabeled graphs $C_i \in \mathcal{V}_d$. Define $\bm{\alpha}(G)$ to be the point in tropical space recording the exponents in the factorization of $G$:
$$\bm{\alpha} (G)=\sum_{i=1}^s \alpha_i\mathbf{1}_{G_i}.$$
\end{definition}

\begin{example}
Consider $\mathcal{V}_1 = \{ \uvedge, \uHtwo \}$.
Then $\bm{\alpha}(\uvedge^3)=(3,0)$, $\bm{\alpha}(\uHtwo\uvedge^2)=(2,1)$.
\end{example}

\begin{definition}
For a pair of partially labeled graphs $A, B$,
define $$m(A,B)=\bm{\alpha}([[A^2]][[B^2]])-\bm{\alpha}([[AB]]^2).$$
\end{definition}

\begin{example}
Consider $A=\pthree{1}{2}{3}{4}\uvedge$ and $B = \pfour{1}{2}{3}{4}{}$. Then
$$[[A^2]]=\upthree\uvedge^2, \,\,\,\,\,\,[[B^2]]=\ulongbroom \,\,\, \textup{ and } \,\,\, [[AB]]=\upfour \uvedge.$$
So $$m(A,B)=\mathbf{1}_{\smallupthree} + \mathbf{1}_{\smallulongbroom}- 2\cdot \mathbf{1}_{\smallupfour}.$$
\end{example}

\noindent \emph{Proof strategy for Theorem~\ref{thm:noksos}.}
We need to show that for every $d\geq 1$, the inequality
\begin{equation}\label{eqn:not-valid}
(\oversl{H}^k-\undersl{H}^{k+1})(\mathbf{x}) \geq 0
\end{equation}
is not valid for $\mathcal{S}_d$. Suppose \eqref{eqn:not-valid} is not valid, then for each fixed $d \geq 1$, we {have} a point
$\mathbf{z} \in \trop(\mathcal{S}_d)$ such that $\langle \mathbf{z}, \bm{\alpha}(\oversl{H}^k) -\bm{\alpha}(\undersl{H}^{k+1})\rangle <0$.
 Since $\trop(\mathcal{S}_d)=\textup{cl}(\textup{conv}(\log(\mathcal{S}_d)))$ and the inequality is strict, we can assume that $\mathbf{z}
\in \log(\mathcal{S}_d)$. This means there exists $\mathbf{x} \in \mathcal{S}_d$ such that
$z_C=\log(x_C)$ for all $C\in \mathcal{V}_d$. Exponentiating, we get that $(\oversl{H}^k-\undersl{H}^{k+1})(\mathbf{x})   < 0$. Since for each $d \geq 1$ there is such a $\mathbf{x}$ and $\mathbf{z}$, it follows that
$\oversl{H}^k-\undersl{H}^{k+1}$ is not sos-testable and we are done.

Thus our task is to show that \eqref{eqn:not-valid} is not a valid constraint for $\trop(\mathcal{S}_d)$. For the sake of contradiction, assume it is valid for some $d$, or equivalently that $\bm{\alpha}(\oversl{H}^k) -\bm{\alpha}(\undersl{H}^{k+1})$ is in the dual cone to $\trop(\mathcal{S}_d) = \log(\mathcal{S}_{\tilde{\mathcal{M}}_d}^{2 \times 2})$.
 By Corollary~\ref{cor:dual cone}, the dual cone is generated by the vectors
 $\{m(A,B): A, B \in \tilde{\mathcal{B}}_d \}$. Thus
 \begin{equation}\label{eqn:2}
 \bm{\alpha}({\oversl{H}}^k)-\bm{\alpha}(\undersl{H}^{k+1})=\sum_{A,B \in \tilde{\mathcal{B}}_d} \lambda_{A,B} m(A,B)
 \end{equation}
 where $\lambda_{A,B}\geq 0$. We will now proceed in steps to derive a contradiction.

\medskip

We first exhibit a point $\mathbf{y} \in \trop(\mathcal{S}_{d})$ such that $\langle \mathbf{y}, \bm{\alpha}({\oversl{H}^k})- \bm{\alpha}({\undersl{H}^{k+1}})\rangle$ is small. We prove this via Lemma~\ref{lemma:yintrop} and Lemma~\ref{lem:2kgeneral}.

\begin{definition}
For an unlabeled graph $F$, let $\delta_i(F)$ be the degree of vertex $i$. Define $L(F):=\sum g(\delta_i(F))$ where $g:\mathbb{R}_{\geq 0} \rightarrow \mathbb{R}_{
\leq 0}$ such that $g(0)=0$, $g$ is non-increasing and $g$ is convex.
\end{definition}

\begin{lemma}\label{lemma:yintrop}
The point $\mathbf{y}=( L(C))_{C \in \mathcal{V}_d}$ lies in $\operatorname{trop}( {\mathcal{S}}_d)$.
\end{lemma}
\begin{proof}
Since the dual cone to $\trop(\mathcal{S}_{d})$
is spanned by
$\{m(A,B): A, B\in \tilde{\mathcal{B}}_d\}$, it is enough to check that  $\langle \mathbf{y}, m(A,B)\rangle \geq 0$ for all
$A, B\in \tilde{\mathcal{B}}_d$. We look carefully at $\langle \mathbf{y}, m(A,B)\rangle$.

Since $y_C = L(C) = \sum g(\delta_i(C))$, it suffices to understand the contribution of different types of 
vertices to $\langle \mathbf{y}, m(A,B)\rangle$.
An unlabeled vertex of $A$ (resp. $B$) of degree $d$ leads to two vertices in $A^2$ (resp. $B^2$) both of degree $d$, and one vertex of degree $d$ in $AB$. Such a vertex contributes $g(d)+g(d)-2 g(d)=0$ to $\langle \mathbf{y}, m(A,B)\rangle$.

 If a label is used in only one graph, say in $A$, and the vertex with this label has $s$ fully labeled edges and $t$ partially labeled edges, then in $AB$ we get a vertex of degree $s+t$, while in $A^2$ we get a vertex of degree $s+2t$ and this vertex has no impact on $B^2$. The total contribution of such a vertex is thus $g(s+2t)-2g(s+t)$. Note that $s+2t\leq 2(s+t)$ so $g(s+2t) \geq g(2(s+t))\geq 2g(s+t)$ where the first inequality follows since $g$ is non-increasing and the second follows from the convexity of $g$ and $g(0)=0$.

The last case is if a label is used in both graphs. Suppose that in $A$ it is adjacent to $t_1$ partially labeled edges, $s_1$ fully labeled edges that are also in $B$ and $u_1$ fully labeled edges that are not in $B$. Similarly, suppose that this vertex in $B$ is adjacent to $t_2$ partially labeled edges, $s_2$ fully labeled edges that are also in $A$ and $u_2$ fully labeled edges that are not in $A$. Note that, by definition, $s_1=s_2$. Then in $A^2$ we get a vertex with degree $2t_1+s_1+u_1$, and in $B^2$, a vertex with degree  $2t_2+s_2+u_2$. In $AB$, we get a vertex of degree $t_1+t_2+s_1+u_1+u_2$. The total contribution of such a vertex to $\langle \mathbf{y}, m(A,B)\rangle$ is thus

\begin{align*}
g(2t_1+s_1+u_1)&+g(2t_2+s_1+u_2)-2g(t_1+t_2+s_1+u_1+u_2)\\
&\geq 2g(t_1+t_2+s_1+u_1/2+u_2/2)-2g(t_1+t_2+s_1+u_1+u_2) \geq 0.
\end{align*}
\end{proof}

From now on, we let $g$ be given by $g(m)=-m$ for $0\leq m \leq p+\frac{3}{2}$ and $g(m)=-(p+\frac32)$ for $m > p+\frac32$. Note that $g(0)=0$, and $g$ is convex and non-increasing. This $g$ has the following effect on $L(F)$ for an unlabeled graph $F$: 
if the maximum degree of a vertex in $F$ is $p+1$, then $g(\delta_i(F)) = - \delta_i(F)$ and hence $L(F) = \sum g(\delta_i(F)) = - 2|E(F)|$.

\begin{lemma}\label{lem:2kgeneral}
We have $\langle \mathbf{y}, \bm{\alpha}(\oversl{H}^k)-\bm{\alpha}(\undersl{H}^{k+1})\rangle =|E(\undersl{H})|$.
\end{lemma}

\begin{proof}
By definition of $\oversl{H}$ and $\undersl{H}$, both have the same number of edges, and each vertex has degree at most $p+1$. 
Therefore, $\langle \mathbf{y}, \bm{\alpha}(\oversl{H}^k)\rangle=-2k|E(\oversl{H})|$. Similarly,  
$\langle \mathbf{y}, \bm{\alpha}(\undersl{H}^{k+1})\rangle =-2(k+1)|E(\undersl{H})|$, and the result holds.
\end{proof}

We know that any connected component of $\oversl{H}$ is a trivial square and that each appears only once in $\oversl{H}$ since $\oversl{H}$ is a trivial square. Since $\oversl{H}$ and $\undersl{H}$ are distinct, there must exist a connected component $C$ in $\oversl{H}$ not in $\undersl{H}$. 

\begin{lemma}\label{lem:positivekgeneral}
For every $A, B\in \tilde{\mathcal{B}}_d$ such that the $C$th coordinate of $m(A,B)$ is positive, $\langle \mathbf{y}, m(A,B)\rangle\geq \frac{1}{2} (z_A + z_B) >0$ where $z_{A}$ is the number of {fully} labeled copies of $C$ that appear in $A$ but not in $AB$, and $z_B$ is the number of {fully}  labeled copies of $C$ that appear in $B$ but not in $AB$. Moreover, we also have $\langle {\bf{1}}_C, m(A,B)\rangle \leq z_A+z_B$.
\end{lemma}

\begin{proof}
If the $C$th component of $m(A,B)$ is positive, then this means that $[[A^2]][[B^2]]$ must contain at least one copy of $C$. Since $C$ is a trivial square, at least $A$ or $B$ must contain either an unlabeled copy of $C$ or a fully labeled copy of $C$.

Suppose $A$ contains $l_A$ fully labeled copies of $C$ that appear in $AB$ but not in $B$, and similarly, suppose $B$ contains $l_B$ fully labeled copies of $C$ that appear in $AB$ but not in $A$. Suppose there are $l_{AB}$ fully labeled copies of $C$ that appear in both $A$ and $B$ (and thus also in $AB$). Suppose there are $z_A$ fully labeled copies of $C$ that appear in $A$ that do not appear in $AB$, and similarly, $z_B$ fully labeled copies of $C$ that appear in $B$ but not in $AB$. Finally, suppose $A$ and $B$ respectively contain $u_A$ and $u_B$ unlabeled copies of $C$. Then the $C$th component of $\bm{\alpha}([[A^2]][[B^2]])$ is $l_A+l_{AB}+z_A+2u_A + l_B + l_{AB} + z_B + 2u_B$ and the $C$th component of $\bm{\alpha}([[AB]]^2)$ is $2(l_A+l_{AB}+l_B+u_A +u_B)$. Thus the $C$th component of $m(A,B)$ is $z_A+z_B-l_A-l_B$, and we have $\langle {\bf{1}}_C, m(A,B)\rangle \leq z_A+z_B$ proving the second claim in the lemma.

To complete the proof of the first claim, we can assume that $z_A+z_B\geq 1$ since the $C$th component of $m(A,B)$ is assumed to be strictly positive. Without loss of generality, assume $z_A\geq 1$, i.e., there is a { fully} labeled copy of $C$ in $A$, say $C_l$ with labels $1, 2, \ldots, r$, that does not appear in $AB$. For this to be the case, $B$ must contain at least some labeled vertex $b\in \{1, 2, \ldots, r\}$ that is adjacent to some partially or fully labeled edge not in $C_l$.

Recall that from Lemma \ref{lemma:yintrop}, $\langle \mathbf{y}, m(A,B)\rangle \geq 0$ and it can be obtained as a sum of contributions of different vertices separately, each of which is nonnegative. We show that $b$ will contribute at least $\frac{1}{2}$ to $\langle \mathbf{y}, m(A,B)\rangle$.

Indeed, in $A$, we know that $b$ is adjacent to $t_1=0$ partially labeled edges and suppose it is adjacent to $s_1$  fully labeled edges that are also in $B$ and $u_1$ fully labeled edges that are not in $B$, where $s_1+t_1 \in \{p, p+1\}$ by definition of $\oversl{H}$. Further, in $B$, suppose $b$ is adjacent to $t_2$ partially labeled edges, $s_2=s_1$ fully labeled edges that are also in $A$ and $u_2$ fully labeled edges that are not in $A$. We know that $t_2+u_2\geq 1$. Note that this implies that $b$ in $AB$ will always contribute at least one more edge than in $A$. So the contribution of $b$ to  $\langle \mathbf{y}, m(A,B)\rangle$ is at least $g(2t_1+s_1+u_1)+g(2t_2+s_1+u_2)-2g(t_1+t_2+s_1+u_1+u_2)=g(s_1+u_1)+g(2t_2+s_1+u_2)-2g(t_2+s_1+u_1+u_2)$. Let's consider a few cases.

Either $b$ in $B$ is adjacent to at least $p+2$ edges, i.e., $2t_2+s_1+u_2\geq p+2$ and $g(2t_2+s_1+u_2)=-(p+\frac{3}{2})$, but $b$ in $AB$ is adjacent to at most $p+1$ edges. In that case, $b$ in $A$ can only be adjacent to $p$ edges, i.e., $s_1+u_1=p$ since $u_2+t_2\geq 1$. So $b$ in $AB$ is adjacent to exactly $p+1$ edges and the contribution of $b$ to $\langle \mathbf{y}, m(A,B)\rangle$ is $-p -(p+\frac{3}{2})-2(-(p+1))=\frac{1}{2}$.

If $b$ in $B$ is adjacent to at least $p+2$ edges and $b$ in $AB$ is also adjacent to at least $p+2$ edges, then the contribution is at least $-(p+1)-(p+\frac{3}{2})-2(-(p+\frac{3}{2}))=\frac{1}{2}$.

Otherwise, if $b$ in $B$ is adjacent to at most $p+1$ edges and $b$ in $A$ is adjacent to $p$ edge, then $b$ in $AB$ is adjacent to at least $p+1$ edges, so the contribution is at least $-p-(p+1)-2(-(p+1))=1$. On the other hand, if $b$ in $B$ is adjacent to at most $p+1$ edges and $b$ in $A$ is adjacent to $p+1$ edges, then $b$ in $AB$ is adjacent to at least $p+2$ edges, so the contribution is at least $-(p+1)-(p+1)-2(-(p+\frac{3}{2}))=1$.

Finally, we know that every other vertex of $C_l$ in $A$ contributes at least zero. Thus the contribution of $C_l$ to $\langle \mathbf{y}, m(A,B)\rangle$ is at least $\frac{1}{2}$.

Note that the same argument holds for every of the $z_A+z_B$ fully labeled copy of $C$ in $A$ or $B$ that do not appear in $AB$. Thus, we get a total contribution of at least $\frac{1}{2}(z_A+z_B)$ to $\langle \mathbf{y}, m(A,B)\rangle$.
\end{proof}

\begin{example}
Consider $C = \upthree$ which is a connected component of $\oversl{H} = \upthree$ but not 
$\undersl{H} = \uvedge^3$. Here $p+1 = 2$, and so $g(1) = -1, g(2) = -2$ and $g(3) = -2.5$. 
We saw that 
$$ \textup{if } A=\pthree{1}{2}{3}{4}\uvedge \textup{ and } 
B=\pfour{1}{2}{3}{4}{}, \textup{ then } m(A,B)=\mathbf{1}_{\smallupthree} + \mathbf{1}_{\smallulongbroom}- 2\cdot \mathbf{1}_{\smallupfour}.$$ 
The component indexed by $C$ in $m(A,B)$ is positive, and 
$$y_{\smallupthree}=-6, \,\,\,\,\,\, y_{\smallulongbroom}=-9.5 \,\,\, \textup{ and } \,\,\, y_{\smallupfour}=-8.$$ 
Thus, as proved in Lemma~\ref{lem:positivekgeneral}, 
$$\langle \mathbf{y}, m(A,B)\rangle = -6-9.5-2(-8)=0.5>0.$$
\end{example}

\smallskip 

\begin{proof}[Proof of Theorem~\ref{thm:noksos}]
We will contradict Equation~\eqref{eqn:2} by a simple counting argument. 

Let $C$ be a connected component of $\oversl{H}$ that does not appear in $\undersl{H}$. 
Equating the coordinate indexed by $C$ on both sides of Equation~\eqref{eqn:2},  we get
$$k\leq \sum_{(A,B)\in \mathcal{I}^+} \lambda_{AB}\langle m(A,B), \mathbf{1}_{C}\rangle,$$ 
where $\mathcal{I}^+$ indexes the pairs $(A,B)$ such that $A,B\in \tilde{\mathcal{B}}_d$ and $\langle \mathbf{1}_{C}, m(A,B)\rangle >0$.  
From Lemma~\ref{lem:positivekgeneral}, we know that $\langle \mathbf{1}_{C}, m(A,B)\rangle\leq z_A+z_B$, so 
$$k\leq \sum_{(A,B)\in \mathcal{I}^+} \lambda_{AB}(z_A+z_B).$$

Recall that $\langle \mathbf{y}, \bm{\alpha}({\oversl{H}}^k)-\bm{\alpha}(\undersl{H}^{k+1})\rangle=|E(\undersl{H})|$ from Lemma \ref{lem:2kgeneral}, so 
\begin{align*}
|E(\undersl{H})|&=\left\langle \mathbf{y}, \sum_{A,B\in \tilde{\mathcal{B}}_d} \lambda_{AB} m(A,B)\right\rangle \\
&\geq \sum_{(A,B) \in \mathcal{I}^+}  \lambda_{AB} \langle \mathbf{y},  m(A,B) \rangle\\
& \geq \sum_{(A,B)\in \mathcal{I}^+} \lambda_{AB}\cdot \frac{1}{2}(z_A+z_B)\\
& \geq \frac{1}{2}k
\end{align*}
where the second line follows from the first because $\langle \mathbf{y}, m(A,B) \rangle \geq 0$ and $\lambda_{AB}\geq 0$ for all $A,B\in \tilde{\mathcal{B}}_d$; the third follows from the second by Lemma \ref{lem:positivekgeneral}; and the last implication follows from $\sum_{(A,B)\in \mathcal{I}^+} \lambda_{AB}(z_A+z_B) \geq k$. Therefore, if Equation~\eqref{eqn:2} holds, then 
$k \leq 2 | E(\undersl{H}) |$. This implies that if 
$k\geq 2|E(\undersl{H})|+1$ as assumed in Theorem \ref{thm:noksos}, then Equation~\eqref{eqn:2} cannot hold, completing the proof 
of the theorem.
\end{proof}

From the third section of \cite{BRST18}, we know that the existence of a sos certificate in the gluing algebra we presented here is equivalent to the existence of a sos certificate in Lov\'asz-Szegedy's gluing algebra \cite{lovaszszegedy}, Hatami-Norine's gluing algebra \cite{HN11} and Razborov's flag algebra \cite{Razborov07}. Since the existence of a rational sos certificate for some graph combination $a$ relies on the existence of two 
sos $b$ and $c$ such that $ab=c$, we get the following corollary.

\begin{corollary}
The binomial graph combinations 
$\oversl{H}^k-\undersl{H}^{k+1}$ for $k\geq 2|E(\oversl{H})|+1$, and $\oversl{H}-\undersl{H}$, satisfying the conditions in Theorem~\ref{thm:noksos} 
are not rational sos in Lov\'asz-Szegedy's gluing algebra, Hatami-Norine's gluing algebra or Razborov's flag algebra.
\end{corollary}

Note that $\oversl{H}-\undersl{H}$ needs to be translated to induced densities first in the case of Razborov's flag algebra.

\bibliographystyle{abbrv}
\bibliography{references}

\end{document}